\documentclass[11pt]{article}
\usepackage{amssymb, amsmath, amsthm, amsfonts,xcolor, enumerate, comment}
\usepackage{thmtools}
\usepackage[utf8]{inputenc}
\usepackage{fullpage}
\usepackage{hyperref}
\newtheorem{theorem}{Theorem}[section]
\newtheorem{lemma}[theorem]{Lemma}

\newtheorem{cor}[theorem]{Corollary}

\newtheorem{conj}[theorem]{Conjecture}

\newtheorem{defn}[theorem]{Definition}
\newtheorem{claim}{Claim}
\theoremstyle{definition}
\newcommand\ex{\ensuremath{\mathrm{ex}}}

\newcommand{\vphi}{\varphi}

\newcommand{\cA}{\mathcal{A}}
\newcommand{\cB}{\mathcal{B}}
\newcommand{\cC}{\mathcal{C}}
\newcommand{\cD}{\mathcal{D}}

\newcommand{\cF}{\mathcal{F}}
\newcommand{\cG}{\mathcal{G}}
\newcommand{\cH}{\mathcal{H}}
\newcommand{\cI}{\mathcal{I}}

\newcommand{\cM}{\mathcal{M}}

\newcommand{\cV}{\mathcal{V}}
\newcommand{\fC}{\mathfrak{C}}

\newcommand{\card}[1]{\left| #1 \right|}
\newcommand{\floor}[1]{\left \lfloor #1 \right \rfloor}
\newcommand{\ceil}[1]{\left \lceil #1 \right \rceil}

\newcommand{\MAX}{\mathrm{MAX}}
\newcommand{\OPT}{\mathrm{OPT}}
\newcommand{\OBJ}{\mathrm{OBJ}}
\newcommand{\maxfam}{N_0}
\newcommand{\hmsize}{N_1}
\newcommand{\twomax}{N_2}
\newcommand{\nummax}{M}

\renewcommand{\lg}{\log_2}
\renewcommand{\t}[1]{\tilde{#1}}

\newcommand{\Fq}{\mathbb{F}_q}
\newcommand{\Fqn}{\mathbb{F}_q^n}
\newcommand{\qbinom}[2]{{#1 \brack #2}_q}

\title{Colourings without monochromatic disjoint pairs}
\author{
Dennis Clemens
\thanks{Technische Universit\"at Hamburg, Institut f\"ur Mathematik, Am Schwarzenberg-Campus 3, 21073 Hamburg, Germany. E-mail: {\tt dennis.clemens@tuhh.de}.}
\and
Shagnik Das
\thanks{Freie Universit\"at Berlin, Institut f\"ur Mathematik, Arnimallee 3, 14195 Berlin, Germany. E-mail: {\tt shagnik@mi.fu-berlin.de}.}
\and
Tuan Tran
\thanks{Department of Mathematics, ETH, 8092 Z\"urich. E-mail: {\tt manh.tran@math.ethz.ch}. Supported by the Czech Science Foundation, grant number GJ16-07822Y, and with institutional support RVO:67985807.}
}

\begin{document}

\maketitle

%%%%%%%%%%%%%%%%%%%%%%%%%%%%%%%%%%%%%%%%
%%
%%  SECTION: INTRODUCTION
%%
%%%%%%%%%%%%%%%%%%%%%%%%%%%%%%%%%%%%%%%%

\begin{abstract}
The typical extremal problem asks how large a structure can be without containing a forbidden substructure.  The Erd\H{o}s--Rothschild problem, introduced in 1974 by Erd\H{o}s and Rothschild in the context of extremal graph theory, is a coloured extension, asking for the maximum number of colourings a structure can have that avoid monochromatic copies of the forbidden substructure.

The celebrated Erd\H{o}s--Ko--Rado theorem is a fundamental result in extremal set theory, bounding the size of set families without a pair of disjoint sets, and has since been extended to several other discrete settings.  The Erd\H{o}s--Rothschild extensions of these theorems have also been studied in recent years, most notably by Hoppen, Koyakayawa and Lefmann for set families, and Hoppen, Lefmann and Odermann for vector spaces.

In this paper we present a unified approach to the Erd\H{o}s--Rothschild problem for intersecting structures, which allows us to extend the previous results, often with sharp bounds on the size of the ground set in terms of the other parameters.  In many cases we also characterise which families of vector spaces asymptotically maximise the number of Erd\H{o}s--Rothschild colourings, thus addressing a conjecture of Hoppen, Lefmann and Odermann.
\end{abstract}

\section{Introduction} \label{sec:intro}

The typical extremal question asks how large a structure can be without containing forbidden substructures.  A fundamental example is Mantel's theorem~\cite{M07}, which shows the maximum number of edges in an $n$-vertex triangle-free graph, $\ex(n,K_3)$, is attained by a balanced complete bipartite graph.  Once the basic extremal question has been answered, one may seek to strengthen the result.  This often takes the form of determining the extremal structures, proving stability results below the extremal threshold, or obtaining supersaturation results above the extremal threshold.

Erd\H{o}s and Rothschild~\cite{E74} posed a different extension of Mantel's theorem, asking for the maximum number of two-edge-colourings without monochromatic triangles that an $n$-vertex graph $G$ can have.  If $G$ itself is triangle-free, then certainly any edge-colouring of $G$ will not have any monochromatic triangles, and so one immediately obtains a lower bound of $2^{\ex(n,K_3)}$.  Erd\H{o}s and Rothschild conjectured that this lower bound was tight, and this was confirmed decades later by Yuster~\cite{Y96}.  Alon, Balogh, Keevash and Sudakov~\cite{ABKS04} extended this result to larger cliques.  Tur\'an's theorem~\cite{T41} shows the largest $K_{k+1}$-free $n$-vertex graph is the Tur\'an graph $T_{k}(n)$, and Alon et al. showed $T_{k}(n)$ also maximises the number of two- and three-edge-colourings without a monochromatic $K_{k+1}$.  However, they further showed that this was not the case for four or more colours, with Pikhurko and Yilma~\cite{PY12} later providing some exact results in this setting.  Pikhurko, Staden and Yilma~\cite{PSY16} have recently introduced an asymmetric version of this problem, reducing its asymptotic solution to a large but finite optimisation problem.

While the original Erd\H{o}s--Rothschild problem was phrased in the context of Mantel's theorem, it can be asked for any extremal question: how many $r$-colourings can a structure have without a monochromatic copy of some forbidden substructure?  For instance, Lefmann, Person, R\"odl and Schacht~\cite{LPRS09} studied colourings of $3$-uniform hypergraphs without monochromatic copies of the Fano plane, and Lefmann, Person and Schacht~\cite{LPS10} obtained asymptotic results for arbitrary $k$-uniform hypergraphs.  For more recent results in this line of investigation, see~\cite{HKL12b},~\cite{HKL14} and~\cite{LP13}.

In this paper, we study the Erd\H{o}s--Rothschild problem for intersecting families of sets, vector spaces and permutations, extending the previous results in this direction.  In the following subsections we will review what is already known before presenting our new theorems.

\subsection{Erd\H{o}s--Rothschild for intersecting families}

A $k$-uniform family of sets $\cF \subseteq \binom{[n]}{k}$ is said to be \emph{$t$-intersecting} if $\card{F \cap F'} \ge t$ for any pair of sets $F, F' \in \cF$.  The natural extremal question is to ask how large a $t$-intersecting family can be.  An obvious construction, called the \emph{$t$-star with centre $X$}, is to take all $k$-sets containing some fixed $t$-element subset $X \subseteq [n]$.  Such a family has $\binom{n-t}{k-t}$ sets, and is said to be \emph{trivially} intersecting.  When $t=1$, we simplify notation by simply calling $1$-intersecting families \emph{intersecting}, and dropping the set notation for the centre $x \in [n]$ of \emph{stars}.

The celebrated Erd\H{o}s--Ko--Rado theorem~\cite{EKR61} of 1961 shows that, when $n$ is suitably large in terms of $k$ and $t$, the $t$-stars are the largest $t$-intersecting families in $\binom{[n]}{k}$.  Note that some bound on $n$ is required; for example, if $n \le 2k-t$, $\binom{[n]}{k}$ itself is $t$-intersecting.  Erd\H{o}s, Ko and Rado showed that $2k$ was the correct threshold for stars to become extremal when $t=1$, and the correct threshold for larger $t$ was eventually determined through the efforts of Frankl~\cite{F78} and Wilson~\cite{W84}.

\begin{theorem} \label{thm:ekr}
Suppose $k, t \in \mathbb{N}$ and $n \ge (t+1)(k-t + 1)$.  If $\cF \subseteq \binom{[n]}{k}$ is $t$-intersecting, then $\card{\cF} \le \binom{n-t}{k-t}$.  Moreover, if $n > (t+1)(k-t+1)$, the only $t$-intersecting families of maximum size are the $t$-stars.
\end{theorem}

A cornerstone of extremal set theory, the Erd\H{o}s--Ko--Rado theorem continues to inspire a great deal of research to this day, having been strengthened in several ways.  Researchers have also obtained Erd\H{o}s--Ko--Rado-type statements in various other settings, most notably including vectors spaces and permutations.  We say vector spaces $V_1$ and $V_2$ are $t$-intersecting if $\dim (V_1 \cap V_2) \ge t$, while a pair of permutations $\pi_1$ and $\pi_2$ is $t$-intersecting if they agree on at least $t$ indices; that is, $\card{ \{ i : \pi_1(i) = \pi_2(i) \} } \ge t$.  We shall discuss some of the history of these extensions in greater detail in later sections, but once again the largest $t$-intersecting families are trivial constructions, again called $t$-stars.

In recent years much attention has been paid to the Erd\H{o}s--Rothschild extension of the Erd\H{o}s--Ko--Rado theorem, and this shall be the focus of this paper.  Given a family of sets, vector spaces or permutations, define an \emph{$(r,t)$-colouring} of the family to be an $r$-colouring of its members such that each colour class is $t$-intersecting.\footnote{To simplify the statements of our general results, we do not differentiate between sets, vector spaces or permutations in our notation.  As Poincar\'e noted, ``Mathematics is the art of giving the same name to different things."}  The Erd\H{o}s--Rothschild problem then asks which families maximise the number of $(r,t)$-colourings.

Note that here the forbidden monochromatic structures have size two, namely pairs of sets that intersect in fewer than $t$ elements.  As a result, in contrast to the triangle-free case, the problem is trivial when $r=2$.  Indeed, let $\cF$ be any family, and let $\cF' \subseteq \cF$ be a maximal $t$-intersecting subfamily.  For any $F \in \cF \setminus \cF'$, there must be some $F' \in \cF'$ such that $\{ F, F' \}$ is not $t$-intersecting.  Thus in any $(2,t)$-colouring of $\cF$, $F$ and $F'$ must receive opposite colours.  It follows that every $(2,t)$-colouring of $\cF$ is determined by its restriction to $\cF'$, and hence there are at most $2^{\card{\cF'}}$ colourings.  On the other hand, any two-colouring of a $t$-intersecting family $\cG$ is a $(2,t)$-colouring, giving precisely $2^{\card{\cG}}$ such colourings.  Hence the largest $t$-intersecting families also have the most $(2,t)$-colourings.\footnote{With a little more work, one can often show uniqueness.}

The problem is of interest, then, when $r \ge 3$.  Hoppen, Kohayakawa and Lefmann~\cite{HKL12} determined all $k$-uniform set families with an asymptotically maximum number of $(r,t)$-colourings for all $k, r$ and $t$ and all $n \ge n_0(k,r,t)$ sufficiently large,\footnote{Their bounds seem to require $n= \Omega_{r,t} \left( k^{t^2 + t + 1} \right)$.} and found the exact maximisers whenever $k \ge 2t-1$.  In particular, they showed the optimal families were all unions of $\ceil{ r/3 }$ $t$-stars.  When $t \ge 2$, one still has to specify which $t$-stars to take, as there are non-isomorphic choices, and Hoppen, Kohayakawa and Lefmann gave a precise description of the optimal families.  Some stability results were also obtained, and in a later paper~\cite{HKL15}, the same authors consider the problem where one does not forbid monochromatic disjoint pairs, but rather matchings of larger size.

In~\cite{HLO16}, Hoppen, Lefmann and Odermann studied the Erd\H{o}s--Rothschild problem for intersecting vector spaces.  For a fixed prime power $q$ and integers $k > t \ge 1$, and sufficiently large $n$, they determined which families of $k$-dimensional subspaces of $\Fqn$ maximise the number of $(3,t)$- and $(4,t)$-colourings, and conjectured that for larger $r$, the results should mirror those in the set family setting.

\subsection{Our results}

In this paper we seek to unify and extend the previous results in this direction.  We make use of the work of Balogh, Das, Delcourt, Liu and Sharifzadeh~\cite{BDDLS15}, who provided strong upper bounds on the number of maximal $t$-intersecting families of sets, vector spaces and permutations.  We then give a general framework in which one may couple these bounds with known extremal and stability results to determine which families maximise the number of $(r,t)$-colourings without resorting to arguments specific to the setting under consideration.

For instance, our first theorem provides a simple condition that guarantees the families with the most $(3,t)$-colourings are precisely the largest $t$-intersecting families.  In order to prevent confusion with maximal families, we refer to these largest $t$-intersecting families as being \emph{extremal}.  We shall refer to any family maximising the number of $(r,t)$-colourings as an \emph{optimal} family.

\begin{theorem} \label{thm:3col}
Let $\maxfam$ denote the size of the extremal $t$-intersecting families, $\hmsize$ the size of the largest non-extremal maximal $t$-intersecting families, and suppose two distinct extremal $t$-intersecting families can have at most $\twomax$ members in common. Suppose further that there are at most $\nummax$ maximal $t$-intersecting families. Provided
\begin{equation}\label{ineq:3col}
		\maxfam - \max \left( \hmsize, \twomax \right) - \frac{6 \lg M}{2 \lg 3 - 3} > 0,
\end{equation}
a family $\cF$ can have at most $3^{\maxfam}$ $(3,t)$-colourings, with equality if and only if $\cF$ is an extremal $t$-intersecting family.
\end{theorem}

The inequality~\eqref{ineq:3col} holds in the settings of set families, vector spaces and permutations, and hence one can immediately deduce that the optimal families are precisely the extremal ones.  Moreover, we show that the inequality holds under very mild, and sometimes optimal, conditions on the size of $n$ relative to the other parameters.

We can also apply our methods to study the problem for $r \ge 4$ colours.  In this case, it is known that the optimal families are no longer the extremal $t$-intersecting families, but instead usually a union of $\ceil{r/3}$ extremal families.  This introduces additional complications to the problem, as the choice of which extremal families to take affects the number of $(r,t)$-colourings the resulting family will have.  We provide a general stability result, showing that provided one inequality is satisfied, any family with close to the optimum number of $(r,t)$-colourings must be a union of the right number of extremal $t$-intersecting families.  However, this statement is a little more technical, and hence we shall present later in Section~\ref{sec:multicol}.  Instead, we highlight the corollaries we obtain in the settings of set families and vector spaces.

\begin{restatable}{prop}{rcolsets} \label{prop:rcolsets}
There is some absolute constant $C$ such that whenever $r \ge 5$, $k > t \ge 1$ and $n \ge C r^2 k (k-t)$, the following is true:
\begin{itemize}
	\item[(i)] if $k \ge 2t-1$, a family is optimal if and only if it is the union of $\ceil{r/3}$ $t$-stars with pairwise-disjoint centres, and
	\item[(ii)] if $k \le 2t-2$, all optimal families consist of $\ceil{r/3}$ pairwise-disjoint $t$-stars.
\end{itemize}
\end{restatable}

This characterisation of optimal families was first found by Hoppen, Kohayakawa and Lefmann~\cite{HKL12}, provided $r, k$ and $t$ were fixed and $n$ was sufficiently large.  Our approach using the maximal $t$-intersecting families allows us to give a relatively short proof that holds for much more moderate bounds on $n$.

In the setting of vector spaces, the initial results are due to Hoppen, Lefmann and Odermann~\cite{HLO16}, who determined which families of vector spaces maximise the number of $(r,t)$-colourings when $2 \le r \le 4$.  As we shall describe in detail in Section~\ref{subsec:multicolvectors}, they raised a conjecture about the optimal families for larger values of $r$.  In the following result, we provide a partial characterisation of the asymptotically optimal families.

\begin{restatable}{prop}{multicolvs} \label{prop:multicolvs}
Let $k > t \ge 1$ be integers, $r = 3s$, where $s \ge 2$ is an integer, $q$ a prime power, and $n \ge n_0(r,k,t,q)$ sufficiently large.
\begin{itemize}
	\item[(i)] If $k \le 2t - 1$, then a family $\cV$ of $k$-dimensional subspaces of $\Fqn$ asymptotically maximises the number of $(r,t)$-colourings if and only if $\cV$ is the union of $s$ pairwise-disjoint $t$-stars.
	\item[(ii)] If $k \ge 2t$, then every asymptotically optimal family $\cV$ of $k$-dimensional subspaces of $\Fqn$ is the union of $s$ $t$-stars with centres $T_i$, $1 \le i \le s$, such that $T_i \cap T_j = \{ \vec{0} \}$ for all $i \neq j$.
	\item[(iii)] If $k \ge 2t$ and $q \ge s - 1$, then a family $\cV$ of $k$-dimensional subspaces of $\Fqn$ asymptotically maximises the number of $(r,t)$-colourings if and only if there is some $(2t)$-dimensional subspace $W$ such that $\cV$ is the union of $s$ $t$-stars with centres $T_i$, $1 \le i \le s$, where each $T_i$ is a $t$-dimensional subspace of $W$ and $T_i \cap T_j = \{ \vec{0} \}$ for all $i \neq j$.
\end{itemize}
\end{restatable}

\subsection{Notation and organisation}

Throughout this paper we shall use standard combinatorial notation.  For instance, $[n]$ represents the first $n$ positive integers, $\{ 1, 2, \hdots, n \}$, and shall be the ground set for our set families.  Given a set $S$ and $k \in \mathbb{N}$, $\binom{S}{k}$ represents all $k$-subsets of $S$.  The Gaussian binomial coefficient, $\qbinom{n}{k}$, represents the number of $k$-dimensional subspaces in $\Fqn$.

We shall also consistently label some parameters common to all the settings we shall work in.  $\maxfam$ will denote the size of the extremal $t$-intersecting families, and $\twomax$ will be the largest possible intersection of two distinct extremal families.  $\hmsize$ will represent the size of the largest non-extremal maximal $t$-intersecting families, and $\nummax$ will give an upper bound on the number of maximal $t$-intersecting families.

The remainder of this paper is laid out as follows.  In Section~\ref{sec:3col}, we shall study families with the maximum number of $(3,t)$-colourings, first proving Theorem~\ref{thm:3col}, and then applying it to families of permutations, vector spaces and set families.  We turn our attention to case $r \ge 4$ in Section~\ref{sec:multicol}, where we first prove a general stability result, and then use it to derive Propositions~\ref{prop:rcolsets} and~\ref{prop:multicolvs}.  In Section~\ref{sec:conc} we close with some concluding remarks and open problems.

%%%%%%%%%%%%%%%%%%%%%%%%%%%%%%%%%%%%%%%%
%%
%%  SECTION: THREE COLOURS
%%
%%%%%%%%%%%%%%%%%%%%%%%%%%%%%%%%%%%%%%%%

\section{Three-coloured families} \label{sec:3col}

In this section we prove our results for $(3,t)$-colourings of families.  The first subsection is devoted to the general theorem, which gives a simple condition for the number of $(3,t)$-colourings to be maximised by the largest $t$-intersecting families.  In the subsequent subsections, we verify this condition in the settings of permutations, vector spaces and set families.

\subsection{A general theorem} \label{subsec:3colgen}

We now prove our main result for the three-colour Erd\H{o}s--Rothschild problem.  Theorem~\ref{thm:3col} shows that if a single inequality concerning the number and structure of maximal $t$-intersecting families is satisfied, the number of $(3,t)$-colourings is maximised by the largest $t$-intersecting families.

Despite its generality (or perhaps because of it), Theorem~\ref{thm:3col} has a short and simple proof.

\begin{proof}[Proof of Theorem~\ref{thm:3col}]
First, for every $t$-intersecting family $\cI$, fix an (arbitrary) assignment of a maximal $t$-intersecting family $\cM(\cI)$ containing $\cI$.  Now let $\cF$ be any family, and let $c(\cF)$ denote the number of $(3,t)$-colourings of $\cF$.  We wish to show $c(\cF) \le 3^{\maxfam}$, with equality only if $\cF$ is itself a $t$-intersecting family of size $\maxfam$.

Observe that the colour classes of every $(3,t)$-colouring of $\cF$ give a partition $\cF = \cI_1 \sqcup \cI_2 \sqcup \cI_3$ into $t$-intersecting families.  We can then map the $(3,t)$-colourings of $\cF$ to triples of maximal intersecting families $(\cM_1, \cM_2, \cM_3)$, where $\cM_i = \cM(\cI_i)$ for $1 \le i \le 3$.  Let $c(\cM_1, \cM_2, \cM_3)$ denote the number of $(3,t)$-colourings of $\cF$ mapped to the triple $(\cM_1, \cM_2, \cM_3)$.

The range of this map has size at most $\nummax^3$, and so we can find some triple $(\cM_1, \cM_2, \cM_3)$ with $c(\cM_1, \cM_2, \cM_3) \ge c(\cF) M^{-3}$.  If $\cM_1 = \cM_2 = \cM_3 = \cM$ for some maximal $t$-intersecting family $\cM$, then we have $\cF \subseteq \cM$, and so $c(\cF) = 3^{\card{\cF}} \le 3^{\card{\cM}} \le 3^{\maxfam}$, with equality if and only if $\cF = \cM$ and $\card{\cM} = \maxfam$.  Hence we may assume $\cM_1, \cM_2$ and $\cM_3$ are not all the same.

We now seek an upper bound on the number of $(3,t)$-colourings mapped to $(\cM_1, \cM_2, \cM_3)$.  Noting that a set $F \in \cF$ can receive colour $i$ only if $F \in \cM_i$, we denote by $a(F) = \card{ \{ i : F \in \cM_i \} }$ the number of colours the set $F$ could receive.  For $1 \le j \le 3$, let $a_j = \card{ \{ F \in \cF : a(F) = j \} }$ be the number of sets that can receive $j$ colours.  We then have $c(\cM_1, \cM_2, \cM_3) \le \prod_{F \in \cF} a(F) = 2^{a_2} 3^{a_3}$.

Since there are at least two distinct maximal families in $\cM_1, \cM_2$ and $\cM_3$, we either have a non-extremal $t$-intersecting family or two distinct extremal $t$-intersecting families, and so $a_3 = \card{ \cM_1 \cap \cM_2 \cap \cM_3 } \le \max \left( \hmsize, \twomax \right)$.  We also have $2 a_2 + 3 a_3 \le a_1 + 2a_2 + 3a_3 = \card{\cM_1} + \card{\cM_2} + \card{\cM_3} \le 3 \maxfam$, and so $a_2 \le \frac32(\maxfam - a_3)$.  Thus
\[ c(\cM_1, \cM_2, \cM_3) \le 2^{\frac32 ( \maxfam - a_3)} 3^{a_3} = 2^{\frac32 \maxfam} \left( 3 \cdot 2^{-\frac32} \right)^{a_3} \le 2^{\frac32 \maxfam} \left( 3 \cdot 2^{-\frac32} \right)^{\max \left( \hmsize, \twomax \right)}. \]
Since $c(\cM_1, \cM_2, \cM_3) \ge c(\cF) M^{-3}$, this gives
\[ c(\cF) \le 2^{\frac32 \maxfam} \left( 3 \cdot 2^{-\frac32} \right)^{\max \left( \hmsize, \twomax \right)} M^3 = 3^{\maxfam} \left( 2^{\frac32} 3^{-1} \right)^{\maxfam - \max \left( \hmsize, \twomax \right) - \frac{6 \lg M}{2 \lg 3 - 3}} < 3^{\maxfam}, \]
where the final inequality follows from~\eqref{ineq:3col}.  Hence the only families maximising the number of $(3,t)$-colourings are the extremal $t$-intersecting families.
\end{proof}

Note that Theorem~\ref{thm:3col} actually provides a strong stability result, showing that any family $\cF$ that is not $t$-intersecting must have exponentially fewer than the maximum number of $(3,t)$-colourings, provided the gap in~\eqref{ineq:3col} is large enough.

In order to obtain concrete results for permutations, vector spaces and set families, we must check that~\eqref{ineq:3col} holds.  This will entail using an Erd\H{o}s--Ko--Rado-type theorem to determine $\maxfam$, a Hilton--Milner-type theorem for $\hmsize$, and having an appropriate bound on $\nummax$, the number of maximal intersecting families.  In the following subsections, we verify the inequality in each of these settings.

\subsection{Permutations} \label{subsec:3colperms}

The first setting we consider is that of permutations, which, to the best of our knowledge, has not been studied previously.  Recall that a pair of permutations $\pi_1, \pi_2 \in S_n$ is said to be $t$-intersecting if $\card{ \{ i \in [n] : \pi_1(i) = \pi_2(i) \} } \ge t$; that is, they agree on at least $t$ indices.  A family $\cF \subseteq S_n$ is $t$-intersecting if every pair of its permutations is $t$-intersecting.

The natural construction of a $t$-intersecting family in $S_n$  is the $t$-star, which, for some $t$ indices $i_1 < i_2 < \hdots < i_t$ and $t$ distinct elements $x_1, x_2, \hdots, x_t \in [n]$, consists of all permutations $\pi$ such that $\pi(i_j) = x_j$ for all $1 \le j \le t$.  Such a family has size $(n-t)!$, and Ellis, Friedgut and Pilpel~\cite{EFP11} showed that, for fixed $t$ and $n$ sufficiently large, there are no larger $t$-intersecting families.

Ellis~\cite{E11} further proved a stability result, showing that for $t \in \mathbb{N}$ and $n \rightarrow \infty$, if $\cF \subseteq S_n$ is a $t$-intersecting family that is not contained in a $t$-star, we must have $\card{\cF} \le \left(1 - e^{-1} + o(1) \right) (n-t)!$.  The final ingredient we shall need, a bound on the number of maximal $t$-intersecting families, was obtained by Balogh et al.~\cite{BDDLS15}, who showed that when $n \ge t \ge 1$, there are at most $n^{n 2^{2n-2t + 1}}$ maximal $t$-intersecting families in $S_n$.

Combining these results with Theorem~\ref{thm:3col} allows us to determine which families in $S_n$ maximise the number of $(3,t)$-colourings.

\begin{cor} \label{cor:3perm}
For every $t \ge 1$, there is an $n_0 = n_0(t)$ such that if $n \ge n_0(t)$, then a family $\cF \subseteq S_n$ can have at most $3^{(n-t)!}$ $(3,t)$-colourings, with equality if and only if $\cF$ is a $t$-star.
\end{cor}

\begin{proof}
We choose $n_0(t)$ to be large enough so that the results cited above apply and the inequality below is satisfied.  In particular, the extremal and stability results give $\maxfam = (n-t)!$ and $\hmsize = \left(1 - e^{-1} + o(1) \right) (n-t)!$.  Since two distinct $t$-stars are either disjoint or fix at least $t+1$ elements, it follows that $\twomax = (n-t-1)!$.  Finally, the bound on the number of maximal families allows $\nummax = n^{n2^{2n-2t+1}}$.  With these parameters in place, it remains to verify that~\eqref{ineq:3col} holds.  We have, for large $n$,
\[ \maxfam - \max \left( \hmsize, \twomax \right) - \frac{6 \lg \nummax}{2 \lg 3 - 3} \ge \frac{(n-t)!}{2e} - 36 \cdot 2^{2n - 2t + 1} n \lg n > 0, \]
since $\frac{6}{2 \lg 3 - 3} \le 36$ and $(n-t)! \ge \left( \frac{n-t}{e} \right)^{n-t}$.  Hence, by Theorem~\ref{thm:3col}, a family $\cF \subseteq S_n$ can have at most $3^{(n-t)!}$ $(3,t)$-colourings, with equality if and only if $\cF$ is a $t$-star.
\end{proof}

\subsection{Vector spaces} \label{subsec:3colvectors}

We next handle the setting of vector spaces.  For some prime power $q$, consider the $n$-dimensional vector space $\Fqn$.  Note that the number of $k$-dimensional subspaces in $\Fqn$ is given by the Gaussian binomial coefficient
\[ \qbinom{n}{k} = \prod_{i=0}^{k-1} \frac{q^{n-i} - 1}{q^{k-i}-1}. \]

Two $k$-dimensional subspaces $V_1$ and $V_2$ are $t$-intersecting if $\dim( V_1 \cap V_2 ) \ge t$, and a family $\cV$ of subspaces is $t$-intersecting if every pair of its subspaces is.  Again, the natural construction of a $t$-intersecting family of $k$-dimensional subspaces in $\Fqn$ is a $t$-star, which consists of all $k$-dimensional subspaces containing some fixed $t$-dimensional subspace.  Such a family contains $\qbinom{n-t}{k-t}$ subspaces.

For $t = 1$, Hsieh~\cite{H75} showed that the stars are the largest intersecting families when $n \ge 2k+1$.  This result was then extended by Frankl and Wilson~\cite{FW86} to all larger $t$, who proved that the same bound on $n$ guarantees that the $t$-stars are the largest $t$-intersecting families.

Stability results for these extremal theorems have also been obtained.  When $t = 1$, Blokhuis, Brouwer, Chowdhury, Frankl, Mussche, Patk\'os and Sz\H{o}nyi~\cite{BBCFMPS10} showed that if $q \ge 3$ and $n \ge 2k+1$, or $q = 2$ and $n \ge 2k+2$, the largest intersecting family that is not contained in a star has size $\qbinom{n-1}{k-1} - q^{k(k-1)} \qbinom{n-k-1}{k-1} + q^k$.  For larger $t$, a stability result was obtained by Ellis~\cite{E16}.  He proved that for fixed $k, q$ and $t$ and sufficiently large $n$, the largest $t$-intersecting family not contained in a $t$-star has size $\left(1 + O(q^{-n}) \right) \qbinom{t+2}{1} \qbinom{n-t-1}{k-t-1}$ if $k \le 2t+1$, and size $\left(1 - O ( q^{-n} ) \right) \qbinom{k-t+1}{1} \qbinom{n-t-1}{k-t-1}$ if $k \ge 2t+2$.

Finally, Balogh et al.~\cite{BDDLS15} bounded the number of maximal intersecting families of vector spaces by $\qbinom{n}{k}^{\binom{2k-1}{k-1}}$.  However, a simple modification of their proof, using F\"uredi's $t$-intersecting version of the Bollob\'as theorem for pairs of vector spaces (see~\cite{F84}), shows that the number of maximal $t$-intersecting families of vector spaces can be bounded by $\qbinom{n}{k}^{\binom{2(k-t)+1}{k-t}}$.

Using these results in unison with Theorem~\ref{thm:3col}, we can determine which families of $k$-dimensional subspaces of $\Fqn$ maximise the number of $(3,t)$-colourings.

\begin{cor} \label{cor:3vector}
A family $\cV$ of $k$-dimensional subspaces of $\Fqn$ can have at most $3^{\qbinom{n-t}{k-t}}$ $(3,t)$-colourings, with equality if and only if $\cV$ is a $t$-star, provided one of the following holds:
\begin{itemize}
	\item[(i)] $t \ge 2$, and $n$ is sufficiently large with respect to $q, t$ and $k$.
	\item[(ii)] $t = 1, k \ge 2$ and
\[ n \ge \begin{cases}
	2k + 1 &\mbox{ if } k = 2 \mbox{ and } q \ge 16, k = 3 \mbox{ and } q \ge 4, \mbox{or } k \ge 4 \mbox{ and } q \ge 3. \\
	2k + 2 &\mbox{ if } k \ge 4 \mbox{ and } q = 2. \\
	14 &\mbox{ otherwise}. \end{cases} \]
\end{itemize}
\end{cor}

Note that when $t = 1$, the bounds of $n \ge 2k + 1$ for $q \ge 3$ and $n \ge 2k + 2$ for $q = 2$ are tight, as the stars are not the unique extremal families for smaller values of $n$.  When $t \ge 2$, we require $n$ to be large so that we may apply the stability result of Ellis~\cite{E16}.  Ellis conjectured that his result should hold for all $n \ge 2k + 1$, which, if true, would allow us to obtain similarly tight bounds for all $t$.

\begin{proof}[Proof of Corollary~\ref{cor:3vector}]
Before we begin our calculations, it will be useful to have some bounds on the Gaussian binomial coefficient.  Observe that
\begin{equation} \label{ineq:qbounds}
 q^{k(n-k)} \le \qbinom{n}{k} = \prod_{i=0}^{k-1} \frac{q^{n-i}-1}{q^{k-i}-1} = q^{k(n-k)} \prod_{i=0}^{k-1} \frac{q^{k-i} - q^{k-n}}{q^{k-i} - 1} \le q^{k(n-k)} \prod_{i=0}^{k-1} \frac{q^{k-i}}{q^{k-i} - 1} \le 4 q^{k(n-k)}.
\end{equation}

We first handle the case when integers $2 \le t < k$ and a prime power $q$ are fixed, and $n$ is sufficiently large.  We have $\maxfam = \qbinom{n-t}{k-t}$ and $\hmsize = \left(1 + O(q^{-n}) \right) \max \left( \qbinom{t+2}{1}, \qbinom{k-t+1}{1} \right) \qbinom{n-t-1}{k-t-1}$ from the previously cited results.  The intersection of two distinct $t$-stars fixes at least a $(t+1)$-dimensional subspace, and hence $\twomax = \qbinom{n-t-1}{k-t-1}$.  We therefore have $\max \left( \hmsize, \twomax \right) \le 2 \qbinom{k+1}{1} \qbinom{n-t-1}{k-t-1}$.  Finally, as discussed above, the number of maximal $t$-intersecting families can be bounded by $\nummax = \qbinom{n}{k}^{\binom{2(k-t)+1}{k-t}}$.

To settle this case, we verify~\eqref{ineq:3col} holds.  Indeed,
\begin{align*}
\maxfam - \max \left( \hmsize, \twomax \right) - &\frac{6 \lg \nummax}{2 \lg 3 - 3} \ge \qbinom{n-t}{k-t} - 2 \qbinom{k+1}{1} \qbinom{n-t-1}{k-t-1} - 36 \binom{2(k-t)+1}{k-t} \lg \qbinom{n}{k} \\
	&= \left( \frac{q^{n-t} - 1}{q^{k-t} - 1} - 2 \qbinom{k+1}{1} \right) \qbinom{n-t-1}{k-t-1} - 36 \binom{2(k-t)+1}{k-t} \lg \qbinom{n}{k} \\
	&\ge \left( q^{n-k} - 4 q^k \right) q^{(k-t-1)(n-k)} - 36 \cdot 4^{k-t} \left( 2 + k (n-k) \lg q \right) > 0
\end{align*}
for $n$ sufficiently large.  Hence, by Theorem~\ref{thm:3col}, the families maximising the number of $(3,t)$-colourings are precisely the $t$-stars.

When $t = 1$, we instead have $\maxfam = \qbinom{n-1}{k-1}$ and $\hmsize = \qbinom{n-1}{k-1} - q^{k(k-1)} \qbinom{n-k-1}{k-1} + q^k$.  The intersection of two stars contains $\twomax = \qbinom{n-2}{k-2}$ subspaces, which is slightly smaller than $\hmsize$.  Hence $\max \left( \hmsize, \twomax \right) = \qbinom{n-1}{k-1} - q^{k(k-1)} \qbinom{n-k-1}{k-1} + q^k$.  Finally, the number of maximal families can be bounded by $\nummax = \qbinom{n}{k}^{\binom{2k-1}{k-1}}$.  Putting these parameters together, we have
\begin{align} \label{calc:3vector}
	\maxfam - \max \left( \hmsize, \twomax \right) - \frac{6 \lg \nummax}{2 \lg 3 - 3} &\ge q^{k(k-1)} \qbinom{n-k-1}{k-1} - q^k - 36 \binom{2k-1}{k-1} \lg \qbinom{n}{k}  \notag \\
	&\ge q^{k(k-1)} \cdot q^{(k-1)(n-2k)} - q^k - 9 \cdot 4^k \left( 2 + k(n-k) \lg q \right) \notag \\
	&= q^{(k-1)(n-k)} - q^k - 9 \cdot 4^k \left( 2 + k(n-k) \lg q \right).
\end{align}

For fixed $k$ and $q$, suppose~\eqref{calc:3vector} is positive for some $n = n_0$.  Increasing $n$ by one can at most double the terms being subtracted, while the positive term increases by a factor of $q^{k-1} \ge 2$.  Hence~\eqref{calc:3vector} must remain positive for all $n \ge n_0$.

Suppose we first wish to determine for which values of $k$ and $q$ it suffices to take $n \ge 2k+1$.  By our above comment, it suffices to check if~\eqref{calc:3vector} is positive when $n = 2k+1$.  Making the substitution, the expression simplifies to
\begin{equation} \label{calc:3vector2}
q^{k^2 - 1} - q^k - 9 \cdot 4^k \left( 2 + k(k+1) \lg q \right).
\end{equation}

Now suppose~\eqref{calc:3vector2} is positive for $k = k_0$.  If we increase $k$ by one, the positive term increases by a factor of at least $q^{2k + 1}$, while the terms being subtracted only increase by factors of at most $q$ and $8$ respectively, and hence~\eqref{calc:3vector2} remains positive for all $k \ge k_0$.

If we take $k = k_0 = 4$,~\eqref{calc:3vector2} further simplifies to $q^{15} - q^4 - 4608 \left( 1 + 10 \lg q \right)$, which is positive for all $q \ge 3$.  Hence, if $k \ge 4$ and $q \ge 3$,~\eqref{calc:3vector} is positive for all $n \ge 2k+1$.  By Theorem~\ref{thm:3col}, it follows that the stars are the only families maximising the number of $(3,1)$-colourings.

The remaining cases all follow from similar calculations.
\end{proof}

\subsection{Set families} \label{subsec:3colsets}

We conclude this section by presenting our results for set families.  The extremal result is the Erd\H{o}s--Ko--Rado theorem~\cite{EKR61}, given in Theorem~\ref{thm:ekr}, which shows that if $n > (t+1)(k-t+1)$, the largest $t$-intersecting set families in $\binom{[n]}{k}$ are the $t$-stars, which have size $\binom{n-t}{k-t}$.

When $t= 1$, a stability result was given by Hilton and Milner~\cite{HM67}, who proved that the largest intersecting family not contained in a star has size $\binom{n-1}{k-1} - \binom{n-k-1}{k-1} + 1$.  The analogous question for larger $t$ was resolved by Frankl~\cite{F78b} and Ahlswede and Khachatrian~\cite{AK96}.  When $n > (t+1)(k-t+1)$, the largest intersecting family not contained in a $t$-star has size
\[ \card{\cF} = \begin{cases}
	\binom{n-t}{k-t} - \frac{n - (t+1)(k-t+1)}{n-t-1} \binom{n-t-1}{k-t} &\mbox{ if } k \le 2t+1, \\
	\binom{n-t}{k-t} - \min \left( \frac{n - (t+1)(k-t+1)}{n-t-1} \binom{n-t-1}{k-t}, \binom{n-k-1}{k-t} - t \right) & \mbox{ if } k \ge 2t+2.
\end{cases} \]

A bound on the number of maximal $t$-intersecting families was given by Balogh et al.~\cite{BDDLS15}, who showed there are at most $\binom{n}{k}^{\binom{2(k-t)+1}{k-t}}$ such families.

Armed with these results, we can use Theorem~\ref{thm:3col} to bound the number of $(3,t)$-colourings of $k$-uniform set families.

\begin{cor} \label{cor:3sets}
There is some absolute constant $n_0$ such that for integers $1 \le t < k$, if $n \ge n_0$ and $n \ge (t+1)(k-t+1)+\eta_{k,t}$, where
\[ \eta_{k,t} = \begin{cases}
	1 & \textrm{if } \min( t, k-t ) \ge 3, \\
	k + 10 \ln k & \textrm{if } t = 1, \\
	10000k & \textrm{if } t = 2 \textrm{ or } k-t \le 2,
\end{cases} \]
a set family $\cF \subseteq \binom{[n]}{k}$ can have at most $3^{\binom{n-t}{k-t}}$ $(3,t)$-colourings, with equality if and only if $\cF$ is a $t$-star.
\end{cor}

\begin{proof}
Before delving into precise calculations, we first show that the inequality~\eqref{ineq:3col} is easily satisfied when $n$ is sufficiently large in terms of $k$ and $t$.  The intersection of two stars has size at most $\binom{n-t-1}{k-t-1}$, which will be a lower-order term.  Hence, using the extremal and stability results cited above, we have $\maxfam - \max (\hmsize, \twomax) = \min\left( \frac{n-(t+1)(k-t+1)}{n-t-1} \binom{n-t-1}{k-t}, \binom{n-k-1}{k-t} - t \right) = \Omega(n^{k-t})$. On the other hand, $\lg \nummax \le \binom{2(k-t)+1}{k-t} \lg \binom{n}{k} = O \left( \lg n \right) = o \left( n^{k-t} \right)$.  Hence~\eqref{ineq:3col} is satisfied, and we deduce that when $n$ is sufficiently large in terms of $k$ and $t$, the $t$-stars are precisely the families that maximise the number of $(3,t)$-colourings.  By more carefully analysing the quantities in question, we shall now show that this same conclusion holds with sharp dependency of $n$ on $k$ and $t$.  Indeed, when $\min( t, k-t) \ge 3$, we cannot reduce our bound on $n$ any further, as then we will have non-trivial $t$-intersecting families that are the same size as the $t$-stars.

We begin with the case $\min(t, k-t) \ge 3$.  Recall that we have $$\maxfam - \max( \hmsize, \twomax) = \maxfam - \hmsize = \min \left( \frac{n - (t+1)(k-t+1)}{n-t-1} \binom{n-t-1}{k-t}, \binom{n-k-1}{k-t} - t \right), $$ which we may bound from below by 
$\frac{n - (t+1)(k-t+1)}{n-t-1} \binom{n-k-1}{k-t} \geq \frac{1}{n} \binom{n-k-1}{k-t}$.  On the other hand, $\lg \nummax \le \binom{2(k-t) + 1}{k-t} \lg \binom{n}{k} \le n \binom{2(k-t)+1}{k-t}$.  Hence we have
\begin{align*}
\maxfam - \max \left( \hmsize, \twomax \right) - \frac{6 \lg \nummax}{2 \lg 3 - 3} &\ge \frac{1}{n} \binom{n-k-1}{k-t} - 36 n \binom{2(k-t)+1}{k-t} \\
	&\ge \left( \frac{1}{n} \left( \frac{n-k-1}{2(k-t)+1} \right)^{k-t} - 36 n \right) \binom{2(k-t)+1}{k-t},
\end{align*}
since $\binom{a}{r} \ge \left( \frac{a}{b} \right)^r \binom{b}{r}$ for $a \ge b$.  We now make some simplifications.  As $3\le t \le k-3$, we have $n \ge (t+1)(k-t+1) + 1 \ge 4 k - 7$, which, together with the fact that $k = t + (k-t) \ge 6$, gives $\frac{9}{20} n> k+1$.  Similarly, $\frac{1}{2}n \ge 2(k-t)+1$.  Hence $\frac{n-k-1}{2(k-t)+1} \ge \frac{11n}{40(k-t)+20} \ge \frac{11}{10}$.

If $k-t \le n^{1/4}$, $\frac{n-k-1}{2(k-t)+1} = \Omega \left(n^{3/4} \right)$.  Since $k-t \ge 3$, it follows that $\frac{1}{n} \left( \frac{n-k-1}{2(k-t)+1} \right)^{k-t} - 36n = \Omega( n^{5/4})$.  Otherwise, if $k - t > n^{1/4}$, we have $\left(\frac{n-k-1}{2(k-t)+1}\right)^{k-t} \ge 1.1^{n^{1/4}}$, and so $\frac{1}{n} \left( \frac{n-k-1}{2(k-t)+1} \right)^{k-t} - 36n = \Omega \left( 1.1^{n^{1/4}} n^{-1} \right)$.  In either case, provided $n$ is larger than some absolute constant, we will have $\maxfam - \max( \hmsize, \twomax ) - \frac{6 \lg \nummax}{2 \lg 3 - 3} > 0$, thus establishing~\eqref{ineq:3col}.

We next handle the case $t = 1$, for which we have $\maxfam - \max( \hmsize, \twomax ) = \binom{n-k-1}{k-1} - 1$ and $\lg \nummax \le \binom{2k-1}{k-1} \lg \binom{n}{k} < n \binom{2k-1}{k-1}$.  It therefore follows that
\begin{align*}
	\maxfam - \max ( \hmsize, \twomax) - \frac{6 \lg \nummax}{2 \lg 3 - 3} &\ge \binom{n-k-1}{k-1} - 36 n \binom{2k-1}{k-1} \\
	&\ge \left( \left( 1 + \frac{n - 3k}{2k-1} \right)^{k-1} - 36 n \right) \binom{2k-1}{k-1}.
\end{align*}
 Note that we may assume $k \ge 4$, since if $k \le 3$ and $t = 1$, then we know from our initial remarks that~\eqref{ineq:3col} is satisfied for large enough $n$.  If $k \le n^{1/2}$, then $\left( 1 + \frac{n-3k}{2k-1} \right)^{k-1} = \Omega(n^{3/2})$, and so~\eqref{ineq:3col} is certainly satisfied for large $n$.

On the other hand, if $k > n^{1/2}$, we note that $n - 3k \ge 10 \ln k$, since $\eta_{k,1} = k + 10 \ln k$.  As $10 \ln k < 2.5(2k-1)$, and $1 + x > e^{x/2}$ for $0 \le x \le 2.5$, we observe that
\[ \left(1 + \frac{n - 3k}{2k-1}\right)^{k-1} > e^{10 (k-1) \ln k / 2(2k-1)} > k^{2.1} = \omega(n), \]
since $\frac{k-1}{2k-1} \ge \frac37$ and $k > n^{1/2}$.  Hence we again see that~\eqref{ineq:3col} is satisfied, provided $n$ is large enough.

Finally, we consider the remaining cases, when $t = 2$ or $k-t \le 2$.  Once more, we have $\maxfam - \max ( \hmsize, \twomax ) = \min \left( \frac{n-(t+1)(k-t+1)}{n-t-1} \binom{n-t-1}{k-t}, \binom{n-k-1}{k-t} - t \right) \ge \frac{n-(t+1)(k-t+1)}{n-t-1} \binom{n-k-1}{k-t}$.  In this range we are taking $\eta_{k,t} = 10000k > (t+1)(k-t)$, as a result of which $n - t - 1 \ge (t+1)(k-t) + \eta_{k,t} > 2(t+1)(k-t)$, and thus $\frac{n-(t+1)(k-t+1)}{n-t-1} = 1 - \frac{(t+1)(k-t)}{n-t-1} > \frac12$.  Putting these bounds together, $\maxfam - \max ( \hmsize, \twomax ) > \frac12 \binom{n-k-1}{k-t}$.

On the other hand, as before, we have $\lg \nummax \le \binom{2(k-t)+1}{k-t} \lg \binom{n}{k} \le H \left( \frac{k}{n} \right)n \binom{2(k-t)+1}{k-t} $, where $H(x) = - x \lg x - (1-x) \lg (1-x)$ is the binary entropy function.  When $k - t = 1$, this gives
\[ \maxfam - \max ( \hmsize, \twomax ) - \frac{6 \lg \nummax}{2 \lg 3 - 3} > \frac12 (n - k - 1) - 108 H \left( \frac{k}{n} \right) n.\]

As $n > \eta_{k,t} = 10000k$, and $108 H(x) < \frac16$ for $x \le \frac{1}{10000}$, we certainly have
\[ \frac12 (n - k - 1) - 108 H \left( \frac{k}{n} \right) n > \frac13 n - \frac16 n = \frac16 n > 0, \]
and thus~\eqref{ineq:3col} is satisfied when $k - t = 1$.

For $k - t \ge 2$, we may upper bound the entropy by $1$.  We get
\[ \maxfam - \max ( \hmsize, \twomax ) - \frac{6 \lg \nummax}{2 \lg 3 - 3} \ge \left( \frac12 \left( \frac{n-k-1}{2(k-t) + 1} \right)^{k-t} - 36 n \right) \binom{2(k-t)+1}{k-t}. \]

If $2 \le k-t \le n^{1/3}$, $\frac{n-k-1}{2(k-t)+1} = \Omega \left( n^{2/3} \right)$, and hence $\left( \frac{n-k-1}{2(k-t)+1} \right)^{k-t} = \Omega \left( n^{4/3} \right)$, and thus this expression is positive when $n$ is suitably large.  Otherwise, note that $\frac{n-k-1}{2(k-t)+1} \ge \frac{n-k}{2k} > 2$ (with quite a lot of room to spare).  Hence, if $k-t > n^{1/3}$, we have $\left( \frac{n-k-1}{2(k-t)+1} \right)^{k-t} > 2^{n^{1/3}} = \omega(n)$, and~\eqref{ineq:3col} is again satisfied for large enough $n$.

This completes the case analysis, and hence the proof of Corollary~\ref{cor:3sets}
\end{proof}

%%%%%%%%%%%%%%%%%%%%%%%%%%%%%%%%%%%%%%%%
%%
%%  SECTION: MULTIPLE COLOURS
%%
%%%%%%%%%%%%%%%%%%%%%%%%%%%%%%%%%%%%%%%%

\section{Multicoloured families} \label{sec:multicol}

In this section we will investigate the number of $(r,t)$-colourings of families for $r \ge 4$, obtaining a general stability result describing the large-scale structure of near-optimal families when a certain condition is met.  We shall then apply this result to the settings of set families and vector spaces and obtain more precise characterisations of the optimal families.  We begin, though, with an optimisation problem that motivates the construction for the lower bound.

\subsection{An optimisation problem} \label{subsec:multicolopt}

We saw in Section~\ref{sec:3col} that the number of $(3,t)$-colourings was maximised by the largest $t$-intersecting families.  When there are more colours available, a largest $t$-intersecting family has $r^{\maxfam}$ $(r,t)$-colourings, since each of the $\maxfam$ members of the family can receive any of the $r$ colours.  However, when $r \ge 4$, we can do better by distributing the colours between a larger number of extremal $t$-intersecting families.  The following optimisation problem, earlier discussed in~\cite{HKL12}, suggests that it is optimal to take $\ceil{r/3}$ extremal $t$-intersecting families and assign three colours to most of them.

\begin{lemma} \label{lem:opt}
Let $r \ge 0$ be an integer.  Consider the maximisation problem below, denoted $\MAX(r)$,
\begin{align*}
\underset{\substack{0 \le s \le r\\ \vec{m} = (m_1, m_2, \hdots, m_s) \in \mathbb{N}^s}}{\textup{maximise}} & \OBJ(\vec{m}) = \prod_{i = 1}^s m_i \\
\textup{subject to} \quad & \sum_{i=1}^s m_i = r,
\end{align*}
and let $\OPT(r)$ denote its optimal value.  The following statements hold.
\begin{itemize}
	\item[(i)] For a feasible vector $\vec{m}$, either $\OBJ(\vec{m}) = \OPT(r)$ or $\OBJ(\vec{m}) \le \frac89 \OPT(r)$.
	\item[(ii)] For $r \ge 2$, all optimal solutions take one of the following forms:
	\begin{itemize}
		\item[(a)] $r \equiv 0 \pmod{3}$: $s = r/3$, and all coordinates of $\vec{m}$ are equal to $3$.
		\item[(b)] $r \equiv 1 \pmod{3}$: $s = \floor{ r/3 }$, with one coordinate of $\vec{m}$ equal to $4$ and all others to $3$, or $s = \ceil{r/3}$, with two coordinates of $\vec{m}$ equal to $2$ and all others to $3$.
		\item[(c)] $r \equiv 2 \pmod{3}$: $s = \ceil{r/3}$, with one coordinate of $\vec{m}$ equal to $2$ and all others to $3$.
	\end{itemize}
\end{itemize}
\end{lemma}

\begin{proof}
We will prove the lemma by showing that, unless $\vec{m}$ is as in (a), (b) or (c) of (ii), we can modify $\vec{m}$ to increase its objective value by a factor of at least $\frac98$.

First suppose $m_i \ge 5$ for some $i$.  Consider a new vector $\vec{m}'$, where we add a new coordinate equal to $3$, and replace $m_i$ by $m_i - 3$.  The vector $\vec{m}'$ is clearly feasible, and we have $\OBJ(\vec{m}') = \frac{3(m_i-3)}{m_i} \OBJ(\vec{m}) \ge \frac65 \OBJ(\vec{m}) > \frac98 \OBJ(\vec{m})$.

Hence we may assume every coordinate is at most $4$.  Now suppose $m_i = 1$ for some $i$, and let $j \neq i$ represent some other coordinate.  Let $\vec{m}'$ be the vector formed by removing the $i$th coordinate, and replacing $m_j$ with $m_j +1$.  $\vec{m}'$ is again feasible, with $\OBJ(\vec{m}') = \frac{m_j + 1}{m_j} \OBJ(\vec{m}) \ge \frac54 \OBJ(\vec{m}) > \frac98 \OBJ(\vec{m})$.

Thus every coordinate must be either $2$, $3$ or $4$.  Replacing every $4$ with two coordinates both equal to $2$ preserves feasibility without changing its objective value.  Suppose now we have at least three coordinates equal to $2$.  Form a new vector $\vec{m}'$ by replacing those three coordinates with two coordinates equal to $3$.  $\vec{m}'$ is still feasible, and $\OBJ(\vec{m}') = \frac98 \OBJ(\vec{m})$.  Hence there can be at most two coordinates equal to $2$, and, up to permutation of coordinates, there is only one option for every $r$.

This implies the optimal solutions have all coordinates equal to $3$, except for perhaps one or two coordinates equal to $2$, or one coordinate equal to $4$, giving the characterisation in $(ii)$.
\end{proof}

\subsection{A structural result} \label{subsec:multicolgen}

We now proceed to the main result of this section, where we will show that, provided a single inequality holds, any family with close to the maximum number of $(r,t)$-colourings must be the union of a given number of extremal $t$-intersecting families.  This rough structural characterisation will allow us to classify the optimal families more precisely when applied to specific settings in later subsections.

We start, though, with a lower bound on the number of $(r,t)$-colourings a family can have.  Recall that $\maxfam$ is the largest $t$-intersecting family, and $\twomax$ is the largest possible intersection of two extremal $t$-intersecting families.  Set $s = \ceil{r/3}$, and let $\cF = \cI_1 \cup \hdots \cup \cI_s$, where each $\cI_i$ is a distinct extremal $t$-intersecting family.

Now let $\vec{m} = (m_1, \hdots, m_s)$ be an optimal solution to $\MAX(r)$, and partition the $r$ colours such that $m_i$ colours are assigned to the family $\cI_i$.  Consider the colourings obtained by colouring each $F \in \cF$ with a colour assigned to one of the families $\cI_i$ containing $F$.  Each such colouring is an $(r,t)$-colouring, as every colour class is contained in some $\cI_i$.  For each $i \in [s]$, there are at least $\maxfam - (s-1) \twomax > \maxfam - \frac{r}{3} \twomax$ members in $\cI_i \setminus \left( \cup_{j \neq i} \cI_j \right)$, each of which has exactly $m_i$ colours available to it.  Hence the number of $(r,t)$-colourings of $\cF$ is at least
\begin{equation} \label{ineq:rcollowbound}
\prod_i m_i^{\maxfam - r \twomax / 3} = \OPT(r)^{\maxfam - r \twomax / 3}.
\end{equation}

In what follows, we shall show that any family with more than three-fourths as many $(r,t)$-colourings as in~\eqref{ineq:rcollowbound} must be the union of extremal $t$-intersecting families.  Moreover, we shall in fact prove that most colourings must be of the above form.  Recall that an $(r,t)$-colouring of a set family can be mapped to a sequence $(\cM_1, \cM_2, \hdots, \cM_r)$, where $\cM_i$ is a maximal $t$-intersecting family containing the $i$th colour class.  Given a sequence of $r$ maximal $t$-intersecting families, let $(\cM'_1, \hdots, \cM'_s)$ denote the distinct families in the sequence, and $\vec{m} = (m_1, \hdots, m_s)$ the multiplicities with which they appear.  The following definition characterises the typical $(r,t)$-colourings of near-optimal families.

\begin{defn}[Typical colourings]
We say an $(r,t)$-colouring of $\cF$ is a \emph{typical colouring} if it maps to a sequence of maximal families where the multiplicity vector $\vec{m}$ is an optimal solution to $\MAX(r)$ and each $\cM'_i$ is an extremal $t$-intersecting family with $\card{ \cF \cap \cM'_i} > r \max ( \hmsize, \twomax )$.

Otherwise we call the colouring an \emph{atypical colouring}.
\end{defn}

Our main theorem now provides a condition that guarantees families close to being optimal are unions of extremal $t$-intersecting families, with the majority of their $(r,t)$-colourings being typical.

\begin{theorem} \label{thm:multicol}
Suppose $r \ge 4$ and $t \in \mathbb{N}$.  Let $\maxfam$ denote the size of the largest $t$-intersecting family, $\hmsize$ the size of the largest non-extremal maximal $t$-intersecting family, and suppose two distinct extremal $t$-intersecting families can have at most $\twomax$ members in common.  Suppose further that there are at least $\ceil{r/3}$ distinct extremal $t$-intersecting families, and at most $\nummax$ maximal $t$-intersecting families.  Define
\begin{equation} \label{defn:multicol}
\Delta = \left( \lg 9 - 3 \right) \maxfam - r \max \left( \hmsize, \twomax \right) - \frac{r \twomax}{3} \lg \OPT(r) - r \lg \nummax.
\end{equation}
If $\Delta \ge 2$, any family $\cF$ with more than three-fourths of the maximum number of $(r,t)$-colourings must be a union of either $\ceil{r/3}$ or, if $r \equiv 1 \pmod{3}$, possibly instead $\floor{r/3}$, extremal $t$-intersecting families.  Furthermore, the proportion of atypical colourings must be smaller than $2^{2-\Delta}/3$.
\end{theorem}

\begin{proof}
The existence of at least $\ceil{r/3}$ distinct extremal $t$-intersecting families implies that the lower bound of~\eqref{ineq:rcollowbound} holds, and hence any family $\cF$ with more than three-fourths the maximum number of colourings must have more than $(3/4) \OPT(r)^{\maxfam - r \twomax / 3}$ $(r,t)$-colourings.

We now seek to describe the structure of families with this many $(r,t)$-colourings.  Given a family $\cF$, let $c(\cM_1, \hdots, \cM_r)$ count the number of $(r,t)$-colourings of $\cF$ mapped to the sequence $(\cM_1, \hdots, \cM_r)$ of maximal $t$-intersecting families.  Recall that we let $(\cM'_1, \hdots, \cM'_s)$ denote the distinct maximal families in the sequence, with $\vec{m} = (m_1, \hdots, m_s)$ recording their multiplicities.  The following claim strongly restricts which sequences can arise from many $(r,t)$-colourings.

\begin{claim} \label{clm:multiplicities}
If, for some $i \in [s]$, $\card{\cF \cap \cM'_i} \le r \max \left( \hmsize, \twomax \right)$, or if $\vec{m}$ is not an optimal solution to $\MAX(r)$, then
\[ c(\cM_1, \hdots, \cM_r) < 2^{-\Delta} M^{-r} \OPT(r)^{\maxfam - r \twomax/3}. \]
\end{claim}

We will soon prove this claim, but first let us see how it implies Theorem~\ref{thm:multicol}.  Taking a union bound over all $\nummax^r$ sequences of maximal families, Claim~\ref{clm:multiplicities} implies the number of atypical colourings is smaller than $2^{-\Delta} \OPT(r)^{\maxfam - r\twomax / 3}$.  Thus, as $\cF$ must have more than $(3/4)\OPT(r)^{\maxfam - r\twomax / 3}$ $(r,t)$-colourings, this implies that fewer than a $2^{2-\Delta}/3$-proportion (and in particular, as $\Delta \ge 2$, fewer than a third) of the $(r,t)$-colourings of $\cF$ can be atypical, as claimed.

Now consider any typical $(r,t)$-colouring of $\cF$.  This gives a partition into colour classes $\cF = \cI_1 \sqcup \hdots \sqcup \cI_r$, where the classes are contained in the extremal $t$-intersecting families $\cM'_1, \hdots, \cM'_s$, appearing with multiplicities given by $\vec{m}$.  As we must have $\card{\cF \cap \cM'_i} > r \max \left( \hmsize, \twomax \right)$, it follows that $\card{ \cM'_i } > \hmsize$, and thus each of the maximal families is in fact extremal.  Moreover, since $\vec{m}$ must be an optimal solution to $\MAX(r)$, we have $s = \ceil{r/3}$, unless $r \equiv 1 \pmod 3$, in which case we could instead have $s = \floor{r/3}$.  This shows that $\cF$ is contained in the right number of extremal $t$-intersecting families.

To finish, we need to show that $\cF$ is in fact equal to the union of those extremal $t$-intersecting families.  We first show that every typical colouring corresponds to the same set of extremal $t$-intersecting families, possibly with different optimal vectors $\vec{m}$.  Suppose for contradiction we had a typical colouring with families $\cM'_1, \hdots, \cM'_{s'}$, and another with families $\cM''_1, \hdots, \cM''_{s''}$, where $\cM''_1 \neq \cM'_i$ for any $1 \le i \le s'$.  Since $\cF$ is contained in the union of the $\cM'_i$, and $s' \le \ceil{r/3}$, we have 
\[ \card{\cF \cap \cM''_1} \le \sum_{i=1}^{s'} \card{\cM'_i \cap \cM''_1} \le s' \twomax < r \twomax, \]
contradicting the definition of a typical colouring.

Hence every typical colouring uses the same set of extremal families.  Now suppose there was some missing member $G \in \cM'_i \setminus \cF$, and let $\tilde{\cF} = \cF \cup \{G\}$.  In every typical colouring of $\cF$ with multiplicity vector $\vec{m}$, there are $m_i \ge 2$ colourings available to members of $\cM'_i$, and hence adding $G$ to $\cF$ at least doubles the number of typical colourings.  While it may be that the atypical colourings of $\cF$ cannot be extended to colour $G$, more than two-thirds of the $(r,t)$-colourings of $\cF$ are typical, and hence $\tilde{\cF}$ has more than four-thirds as many $(r,t)$-colourings as $\cF$, contradicting $\cF$ being close to optimal.

Thus every family $\cF$ close to maximising the number of $(r,t)$-colourings must be the union of $\ceil{r/3}$, (or, when $r \equiv 1 \pmod{3}$, possibly $\floor{r/3}$ instead) extremal $t$-intersecting families.
\end{proof}

It remains to prove Claim~\ref{clm:multiplicities}, a task we now complete.

\begin{proof}[Proof of Claim~\ref{clm:multiplicities}]
For each $F \in \cF$ let  
$a(F) = \card{ \{ j : F\in\cM_j \} } = \sum_{i: F\in\cM'_i} m_i$
denote the number of maximal families containing $F$ in the sequence $(\cM_1, \hdots, \cM_r)$, and thus the number of colours available to $F$.  We then have the bound
\[ c(\cM_1, \hdots, \cM_r) \le \prod_{F \in \cF} a(F) = \prod_{F \in \cF} \left( \sum_{\substack{i \le s:\\ F\in \cM'_i}} m_i \right). \]

We first remove the families with multiplicity $1$ from consideration.  Without loss of generality, suppose there is some $0 \le s' \le s$ such that $m_i \ge 2$ for all $1 \le i \le s'$, and $m_i = 1$ for all $s' < i \le s$.  Suppose $s' < s$, and consider the uniquely-appearing maximal family $\cM'_s$.  Note that any member $F \in \cM'_s$ that is not in $\cM'_i$ for any $i < s$ must have $a(F) = 1$, thus contributing nothing to the upper bound above.  Hence in $\cM'_s$, we need only consider the members belonging to $\cup_{i < s} \left( \cM'_i \cap \cM'_s \right)$.  For every such $F$, we have
\[ a(F) = \sum_{\substack{i \le s :\\ F\in \cM'_i}} m_i = 1 + \sum_{\substack{i < s: \\ F\in \cM'_i}} m_i \le 2 \sum_{\substack{i < s :\\ F\in \cM'_i}} m_i. \]

Moreover, if both $\cM'_i$ and $\cM'_s$ are extremal $t$-intersecting families, then $\card{\cM'_i \cap \cM'_s} \le \twomax$.  Otherwise, at least one of them must be non-extremal, and so $\card{ \cM'_i \cap \cM'_s} \le \hmsize$.  Summing up the $s-1$ intersections, and observing that $s-1 < r$, we collect at most $r \max \left( \hmsize, \twomax \right)$ factors of $2$ in this upper bound, giving
\[ c(\cM_1, \hdots, \cM_r) \le 2^{r \max \left( \hmsize, \twomax \right)} \prod_{\substack{F \in \cF: \\ F \in \cup_{i < s} \cM'_i}} \left( \sum_{\substack{i < s :\\ F\in\cM'_i}} m_i \right). \]

We can then repeat this for each uniquely appearing maximal family, thus eliminating the families $\cM'_{s' + 1}, \hdots, \cM'_s$ from our upper bound.  All the remaining summands then satisfy $m_i \ge 2$, and hence in~\eqref{ineq:sumtoprod} below we can upper bound the inner sums by the corresponding products.
\begin{align}
	c(\cM_1, \hdots, \cM_r) &\le 2^{(s - s') r \max \left( \hmsize, \twomax \right)} \prod_{\substack{F \in \cF: \\ F \in \cup_{i \le s'} \cM'_i} } \left( \sum_{\substack{i \le s' :\\ F\in\cM'_i}} m_i \right) \notag \\
	&\le 2^{(s - s') r \max \left( \hmsize, \twomax \right)} \prod_{\substack{F \in \cF: \\ F \in \cup_{i \le s'} \cM'_i}} \left( \prod_{\substack{i \le s' :\\ F\in\cM'_i}} m_i \right) \label{ineq:sumtoprod} \\
	&= 2^{(s-s') r \max \left( \hmsize, \twomax \right)}\prod_{i=1}^{s'} m_i^{\card{\cF \cap \cM'_i}} \notag \\
	&\le 2^{(s-s') r \max \left( \hmsize, \twomax \right)} \left( \prod_{i=1}^{s'} m_i \right)^{\maxfam}. \label{ineq:multiupbound1}
\end{align}

Now, since we have removed the $s - s'$ uniquely appearing families from consideration, we must have $\sum_{i=1}^{s'} m_i = r - (s - s')$.  The product in the final line is thus bounded by $\OPT(r - (s - s'))$.  Furthermore, the solutions to $\MAX(r)$ given by Lemma~\ref{lem:opt} clearly imply $\OPT(r-1) \le \frac34 \OPT(r)$ for all $r \ge 3$.  As $\OPT(0) = \OPT(1) = 1$, we have $\OPT(r - (s-s')) \le \left( \frac34 \right)^{s - s'} \OPT(r)$, giving
\begin{align}
	c(\cM_1, \hdots, \cM_r) &\le 2^{(s-s') r \max \left( \hmsize, \twomax \right)} \OPT(r - (s-s'))^{\maxfam} \label{ineq:multiupbound2}\\
	&\le \left( \left(\frac34\right)^{\maxfam} 2^{ r \max \left( \hmsize, \twomax \right)} \right)^{s - s'} \OPT(r)^{\maxfam}. \label{ineq:multiupbound3}
\end{align}

With these bounds in place, we can complete the proof of Claim~\ref{clm:multiplicities}.  Observe that, by~\eqref{defn:multicol},
\begin{align*} \lg \left( \left( \frac34 \right)^{\maxfam} 2^{r \max \left( \hmsize, \twomax \right) } \right) &= \left( \lg 3 - 2 \right) \maxfam + r \max \left( \hmsize, \twomax \right) \\
	&< \left( 3 - \lg 9 \right) \maxfam + r \max \left( \hmsize, \twomax \right) \\
	&= - \frac{r \twomax} {3} \lg \OPT(r) - r \lg M - \Delta.
\end{align*}
Hence if $s - s' \ge 1$,~\eqref{ineq:multiupbound3} implies $c(\cM_1, \hdots, \cM_r) < 2^{-\Delta} M^{-r} \OPT(r)^{\maxfam - r\twomax / 3} $.  We may therefore assume $s - s' = 0$.

Now suppose $\card{\cF \cap \cM'_i} \le r \max \left( \hmsize, \twomax \right)$ for some $i$.  The bound in~\eqref{ineq:multiupbound1} is then an overestimate by a factor of at least $m_i^{\maxfam - r \max \left(\hmsize, \twomax \right)}$.  Since $s - s' = 0$, there are no uniquely-appearing families, and so $m_i \ge 2$.  Hence, using~\eqref{defn:multicol},
\begin{align*}
c(\cM_1, \hdots, \cM_r) &\le 2^{r \max \left( \hmsize , \twomax \right) - \maxfam} \OPT(r)^{\maxfam} \\
&< 2^{r \max \left( \hmsize, \twomax \right) - \left( \lg 9 - 3 \right) \maxfam} \OPT(r)^{\maxfam} \\
&= 2^{-\Delta} M^{-r} \OPT(r)^{\maxfam - \frac{r \twomax}{3}}.
\end{align*}

Finally, suppose $\vec{m}$ is not an optimal solution to $\MAX(r)$.  By Lemma~\ref{lem:opt}, we must have $\prod_{i=1}^s m_i = \OBJ(\vec{m}) \le \frac89 \OPT(r)$ in~\eqref{ineq:multiupbound2}.  Hence, using~\eqref{defn:multicol} again, we find
\[ c(\cM_1, \hdots, \cM_r) \le \left( \frac89 \OPT(r) \right)^{\maxfam} = 2^{\left(3 - \lg 9 \right) \maxfam} \OPT(r)^{\maxfam} < 2^{-\Delta} M^{-r} \OPT(r)^{\maxfam - \frac{r \twomax}{3}},\]
as required.  This completes the proof of Claim~\ref{clm:multiplicities}.
\end{proof} 

Theorem~\ref{thm:multicol} provides the desired rough structural characterisation, showing that families close to maximising the number of $(r,t)$-colourings must be unions of extremal $t$-intersecting families.  In fact, if $\Delta$ is large, one obtains a strong stability result: if $\cF$ is not contained in the union of the right number of extremal $t$-intersecting families, then it can only have atypical colourings, whose total number will be insignificant compared to the lower bound $\OPT(r)^{\maxfam - r \twomax / 3}$.

However, unlike Theorem~\ref{thm:3col}, Theorem~\ref{thm:multicol} falls short of determining the optimal families, as the choice of the $\ceil{r/3}$ extremal $t$-intersecting families can affect the number of $(r,t)$-colourings.  In the subsequent subsections we will pursue more precise results for set families and vector spaces, and it will help to have more explicit quantitative bounds on the number of $(r,t)$-colourings of potentially optimal families.

Let $\cF = \cI_1 \cup \hdots \cup \cI_s$, where $s = \ceil{r/3}$, or possibly $\floor{r/3}$ if $r \equiv 1 \pmod 3$, and each $\cI_i$ is an extremal $t$-intersecting family.  The typical $(r,t)$-colourings can be classified by how the colours are partitioned between the families $\cI_i$; let $\vec{C} = (C_1, \hdots, C_s)$ be a partition of $[r]$, where $C_i$ is the set of colours assigned to the family $\cI_i$.  Let $\Phi(\vec{C})$ denote all $(r,t)$-colourings that can arise from this partition; that is, all colourings where each set $F \in \cF$ receives a colour from $\cup_{i : F\in\cI_i} C_i$.  In order for these colourings to be typical, $\vec{m}$, where $m_i = \card{C_i}$, must be an optimal solution to $\MAX(r)$.  Let $\fC$ denote the set of all such partitions $\vec{C}$, and note that $\card{\fC} \le \binom{r}{2} r! / (\prod_i m_i!)$, since we have to choose at most two indices for which $m_i \neq 3$, and then distribute the colours accordingly.

\begin{cor} \label{cor:asymcount}
Under the assumptions of Theorem~\ref{thm:multicol}, if $\cF = \cI_1 \cup \hdots \cup \cI_s$ is a union of extremal $t$-intersecting families, where $s = \ceil{r/3}$, or perhaps $\floor{r/3}$ if $r \equiv 1 \pmod 3$, then the number of $(r,t)$-colourings of $\cF$ is
\[ \left( 1 + O ( 2^{-\Delta} ) \right) \sum_{ \vec{C} \in \fC } \card{\Phi(\vec{C}) }. \]
\end{cor}

\begin{proof}
We begin with some estimates on the quantity $\card{\Phi(\vec{C})}$.  For every choice of $\vec{C}$, the lower bound of~\eqref{ineq:rcollowbound} holds.  On the other hand,~\eqref{ineq:multiupbound3} provides an upper bound.  Hence we have
\[ \OPT(r)^{\maxfam - r \twomax / 3} \le \card{ \Phi( \vec{C} ) } \le \OPT(r)^{\maxfam}. \]

By taking a union bound over all the partitions $\vec{C} \in \fC$, we know that the number of typical colourings of $\cF$ is at most $\sum_{\vec{C} \in \fC} \card{ \Phi( \vec{C} ) }$.  By Theorem~\ref{thm:multicol}, the number of atypical colourings is at most $2^{-\Delta} \OPT(r)^{\maxfam - r \twomax / 3}$.  In light of our lower bound on $\card{\Phi(\vec{C})}$, this shows that the total number of $(r,t)$-colourings of $\cF$ is at most $(1 + 2^{-\Delta}) \sum_{\vec{C} \in \fC} \card{\Phi(\vec{C})}$, giving the required upper bound.

For the lower bound, we only consider the typical colourings arising from a unique partition in $\fC$; that is, the number of $(r,t)$-colourings is at least
\begin{equation} \label{ineq:asymlowbound}
\sum_{\vec{C} \in \fC} \left( \card{\Phi(\vec{C})} - \sum_{\vec{C}' \neq \vec{C}} \card{\Phi(\vec{C}) \cap \Phi(\vec{C}')} \right).
\end{equation}

Given two distinct partitions $\vec{C}$ and $\vec{C}'$, there must be some colour $c$ that is in different parts in the two partitions.  Without loss of generality, $c \in C_1$ and $c \in C'_2$.  Suppose $\card{C_1} \le \card{C_2}$ (the case $\card{C_2} < \card{C_1}$ can be handled similarly).  For a colouring to belong to $\Phi(\vec{C}) \cap \Phi(\vec{C}')$, the colour $c$ can only be used for members of $\cI_1 \cap \cI_2$.  In particular, such a colouring must also belong to $\Phi(\vec{C}'')$, where
\[ C''_i = \begin{cases}
	C_1 \setminus \{c \} &\mbox{if } i = 1 \\
	C_2 \cup \{ c \} &\mbox{if } i = 2\\
	C_i &\mbox{otherwise}
\end{cases}. \]

However, if $m''_i = \card{C''_i}$, then $\vec{m}''$ cannot be an optimal solution to $\MAX(r)$.  Hence, by Claim~\ref{clm:multiplicities}, $\card{ \Phi(\vec{C}) \cap \Phi(\vec{C}') } \le \card{\Phi(\vec{C}'')} \le 2^{-\Delta} \nummax^{-r} \OPT(r)^{\maxfam - r \twomax / 3} \le 2^{-\Delta} \nummax^{-r} \card{\Phi(\vec{C})}$.  This gives, for fixed $\vec{C}$,
\[ \sum_{\vec{C}' \neq \vec{C}} \card{ \Phi(\vec{C}) \cap \Phi(\vec{C}')} < \card{\fC} 2^{-\Delta} \nummax^{-r} \card{\Phi(\vec{C})}. \]

Since we have at least $\ceil{r/3}$ distinct extremal $t$-intersecting families, we must have $\nummax \ge r/3$.  On the other hand, recall that $\card{\fC} \le \binom{r}{2} r! / (\prod_i m_i!)$.  Coupling the upper bound from Stirling's Approximation with the observation that $m_i = 3$ for at least $r/3 - 2$ indices $i$, we have ${\card{\fC} \le 18 r^{5/2}e \cdot (r / (6^{1/3} e) )^r}$.  Hence $\card{\fC} \nummax^{-r} \le 18 r^{5/2} e \cdot ( 3 / (6^{1/3} e) )^r$.  As this expression tends to $0$ as $r \rightarrow \infty$, it can be bounded above by some absolute constant.

Using this upper bound in~\eqref{ineq:asymlowbound}, we find the number of $(r,t)$-colourings of $\cF$ to be at least $(1 - O(2^{-\Delta})) \sum_{\vec{C} \in \fC} \card{\Phi(\vec{C})}$, completing the proof.
\end{proof}

\subsection{Set families} \label{subsec:multicolsets}

We now seek a more precise characterisation of optimal families, and begin in the context of set families.  Using Theorem~\ref{thm:multicol} and Corollary~\ref{cor:asymcount}, we extend the previous results of Hoppen, Kohayakawa and Lefmann~\cite{HKL12} for five or more colours,\footnote{For four colours, our method allows us to determine the optimal families precisely when $t = 1$.  However, for $t \ge 2$, we only get the asymptotics of the maximum number of $(4,t)$-colourings, while the precise characterisation obtained in~\cite{HKL12} (see \cite{HKL13} for the full proof) requires careful study of the atypical colourings as well.} showing they hold even when $n$ is only moderately large in terms of $k$, $t$ and $r$.  We recall our main result from this setting.

\rcolsets*

In order to prove Proposition~\ref{prop:rcolsets}, we require a few preliminary lemmas.  The first of these shows that the quantity $\Delta$ from the previous section can be made arbitrarily large, thus enabling the use of Theorem~\ref{thm:multicol} and Corollary~\ref{cor:asymcount}.

\begin{lemma} \label{lem:setsdelta}
For every $K > 0$ there is some $C = C(K)$ such that if $r \ge 4$, $k > t \ge 1$ and $n \ge C r^2 k (k-t)$, then $\Delta \ge K$, where $\Delta$ is as defined in~\eqref{defn:multicol}.
\end{lemma} 

\begin{proof}
We begin by recalling the values of the key parameters from Section~\ref{subsec:3colsets}.  Provided $C \ge 1$, we certainly have $n > (t+1)(k-t+1)$.  The Erd\H{o}s--Ko--Rado Theorem~\cite{EKR61}, together with the work of Frankl~\cite{F78} and Wilson~\cite{W84}, then establishes that the largest $t$-intersecting families are the $t$-stars, and hence $\maxfam = \binom{n-t}{k-t}$.  As any set in the intersection of two $t$-stars must contain the union of the two centres, which consists of at least $t+1$ elements, we further have $\twomax = \binom{n-t-1}{k-t-1}$.  The number of maximal $t$-intersecting families was bounded by Balogh et al.~\cite{BDDLS15}, who showed that we may take $\nummax = \binom{n}{k}^{\binom{2(k-t)+1}{k-t}}$.

Finally, we are left with the size of the largest $t$-intersecting family not contained in a $t$-star.  Sharpening the bounds of Frankl~\cite{F78b}, Ahlswede and Khachatrian~\cite{AK96} proved $\hmsize = \max( \card{\cH_1}, \card{\cH_2} )$, where $\card{\cH_1} = (t+2)\binom{n-t-2}{k-t-1} + \binom{n-t-2}{k-t-2}$ and $\card{\cH_2} = \binom{n-t}{k-t} - \binom{n-k-1}{k-t} + t$.  Since our $n$ here is quite a bit larger than before, we shall estimate these quantities somewhat differently.  We have $\card{\cH_1} \le (t+2) \binom{n-t-1}{k-t-1}$ and $\card{\cH_2}=\binom{n-t-1}{k-t-1}+\ldots+\binom{n-k-1}{k-t-1}+t \le (k-t+1) \binom{n-t-1}{k-t-1} + t$, and so we will use the bound $\hmsize \le (k+1) \binom{n-t-1}{k-t-1}$.

Substituting these values into~\eqref{defn:multicol} and using the bound $\OPT(r) \le 3^{r/3}$, we find
\begin{align}
&\Delta = (\lg 9 - 3) \maxfam - r \max ( \hmsize, \twomax ) - \frac{r \twomax}{3} \lg \OPT(r) - r \lg \nummax \notag \\
&\ge (\lg 9 - 3 ) \binom{n-t}{k-t} - r (k+1)\binom{n-t-1}{k-t-1} - \frac{r^2 \lg 3}{9} \binom{n-t-1}{k-t-1} - r \binom{2 (k-t) + 1}{k-t} \lg \binom{n}{k} \notag \\
&\ge \left( \lg 9 - 3 - \frac{r(k-t) \left( 9 (k+1) + r \lg 3\right)}{9 (n-t)} \right) \binom{n-t}{k-t} - rk \binom{2(k-t)+1}{k-t} \lg \left( \frac{ne}{k} \right) \notag \\
&\ge \left( \lg 9 - 3 - \frac{k r^2 (k-t) }{n-t} - rk \lg \left( \frac{ne}{k} \right) \left( \frac{2(k-t)+ 1}{n-t} \right)^{k-t} \right) \binom{n-t}{k-t}. \label{ineq:setdelta}
\end{align}

Define $\omega$ so that $n = 2 \omega r^2 k (k-t)$.  Since $t < k$, we have $n - t > \omega r^2 k (k-t)$.  The third term in the parentheses is thus at most $\omega^{-1}$, which will be smaller than $\frac{1}{40}$ provided $\omega \ge 40$.

To bound the last term in the parentheses, first suppose $k - t = 1$.  In this case, we have
\[ r k \lg \left( \frac{ne}{k} \right) \left( \frac{2(k-t) + 1}{n-t} \right)^{k-t} < \frac{3 rk \lg ( 2 e \omega r^2 ) }{\omega r^2 k} = \frac{ 3 \lg \omega + 6 \lg r + 3 \lg (2e ) }{\omega r}, \]
which will be at most $\frac{1}{40}$ if $\omega$ is large enough (independent of $r$ and $k$).

On the other hand, when $k-t \ge 2$,
\[ rk \lg \left( \frac{ne}{k} \right) \left( \frac{2(k-t) + 1}{n-t} \right)^{k-t} \le rk \lg (2 e \omega r^2 k ) \left( \frac{3}{\omega r^2 k} \right)^{k-t} \le \frac{9 \lg (2 e \omega r^2 k )}{\omega^2 r^3 k}, \]
which again can be bounded by $\frac{1}{40}$ when $\omega$ is large (independent of the other parameters).

Hence, provided $\omega$ is large enough, the parenthetical term in~\eqref{ineq:setdelta} is at least $\lg 9 - 3 - \frac{1}{20} \ge \frac{1}{10}$, which implies
\[ \Delta \ge \frac{1}{10} \binom{n-t}{k-t} \ge \frac{1}{10} \omega r^2 k (k-t), \]
which be larger than $K$ provided $\omega > 10 K$.  We may then take $C$ to be the smallest value of $\omega$ for which the claimed bounds all hold.
\end{proof}

The next lemma allows us to compare the number of $(r,t)$-colourings of different unions of $t$-stars.  Recall that $\fC$ denotes the set of all optimal partitions of $r$ colours that give rise to typical $(r,t)$-colourings.

\begin{lemma} \label{lem:setsstarswap}
Suppose $k > t \ge 1$, $r \ge 5$, and $n \ge k + rt$.  Set $s = \ceil{r/3}$ or, if $r \equiv 1 \pmod 3$, possibly $\floor{r/3}$.  Let $\cF = \cI_1 \cup \hdots \cup \cI_s$, where the $\cI_i$ are distinct $t$-stars with centres $T_i$ respectively, and suppose $\card{T_1 \cap T_s} \ge \max (1, 2t - k)$.  Let $\t\cF = \cI_1 \cup \hdots \cup \cI_{s-1} \cup \t\cI_s$, where $\t\cI_s$ is a $t$-star whose centre $\t{T}_s$ is disjoint from $\cup_{i \le s} T_i$.  For $\vec{C} \in \fC$, let $\Phi(\vec{C})$ and $\t\Phi(\vec{C})$ be the corresponding $(r,t)$-colourings of $\cF$ and $\t\cF$ respectively.

We then have $\card{\t\Phi(\vec{C})} \ge \frac65 \card{\Phi(\vec{C})}$, unless $\card{C_1} = \card{C_s} = 2$, in which case $\card{\t\Phi(\vec{C})} \ge \card{\Phi(\vec{C})}$.
\end{lemma}

\begin{proof}
We begin by considering some particular partitions of $\cF$ and $\t\cF$.  Label the elements in the centres of $\cI_s$ and $\t\cI_s$ by setting $T_s = \{ x_1, x_2, \hdots, x_t \}$ and $\t{T}_s = \{ y_1, y_2, \hdots, y_t \}$.  We can then define the bijection $\tau: \cI_s \setminus \t\cI_s \rightarrow \t\cI_s \setminus \cI_s$ by replacing elements in $T_s$ with the corresponding element in $\t{T}_s$; that is, $\tau(F) = F \cup \{ y_i : y_i \notin F\} \setminus \{ x_i : y_i \notin F \}$.  We claim that the bijection $\tau$ has the following two properties.
\begin{itemize}
	\item[(i)] For every $F \in \cI_s \setminus \t\cI_s$ and $i \in [s-1]$, if $\tau(F) \in \cI_i$, then $F \in \cI_i$.
	\item[(ii)] There is some $F \in \cI_1 \cap \cI_s \setminus \t\cI_s$ such that $\tau(F)$ only belongs to $\t\cI_s$.
\end{itemize}

To see why (i) holds, observe that the only elements in $\tau(F) \setminus F$ come from the set $\t{T}_s$, which is disjoint from $\cup_{j \le s} T_j$.  Hence if $\tau(F) \in \cI_i$ --- that is, if $T_i \subset \tau(F)$ --- we must have $T_i \subset F$ as well.  For (ii), let $F$ be a set containing $T_1 \cup T_s$, with the remaining elements in $F$ not belonging to any other centre.  Our lower bound on $n$ guarantees that such a set exists.  The elements of $T_s$ are removed when forming $\tau(F)$, and hence $\tau(F)$ has fewer than $t$ elements from $\cup_{j \le s} T_j$, and is thus not contained in any of the other $t$-stars.

We now use $\tau$ to partition $\cF$ and $\t\cF$ into some subfamilies.  First set $\cG = \cI_1 \cup \hdots \cup \cI_{s-1}$ and $\cA = (\cI_s \cap \t\cI_s) \cup ( \cG \setminus (\cI_s \cup \t\cI_s) )$.  Next define $\cB = \cF \setminus \t\cF = (\cI_s \setminus \t\cI_s) \setminus \cG$, and let $\t\cB = \tau(\cB)$.  Note that Property (i) implies $\t\cB \subset (\t\cI_s \setminus \cI_s) \setminus \cG$.  Let $\t\cC = (\t\cI_s \setminus \cI_s) \cap \cG$, and set $\cC = \tau^{-1}(\t\cC)$.  Again, Property (i) implies $\cC \subset (\cI_s \setminus \t\cI_s) \cap \cG$.  Finally, we denote the remaining sets by $\cD$ and $\t\cD$; that is, $\cD = \{ F \in (\cI_s \setminus \t\cI_s) \cap \cG : \tau(F) \notin \cG \}$, and $\t\cD = \tau(\cD) = \{ F \in (\t\cI_s \setminus \cI_s) \setminus \cG : \tau^{-1}(F) \in \cG \}$.  Note that Property (ii) implies $\cD$ and $\t\cD$ are non-empty.  Observe also that these subfamilies partition $\cF$ and $\t\cF$ in the following ways: $\cF = \cA \sqcup \cB \sqcup \cC \sqcup \t\cC \sqcup \cD$ and $\t\cF = \cA \sqcup \t\cB \sqcup \cC \sqcup \t\cC \sqcup \cD \sqcup \t\cD$.

With these partitions in mind, we prove the lemma by injectively mapping the colourings in $\Phi(\vec{C})$ into (a fraction of) the colourings in $\tilde{\Phi}(\vec{C})$.  For any set $S \subseteq C_1 \cup \hdots \cup C_{s-1}$ of at least two colours, and any non-empty subset $\t{S} \subseteq S$, fix arbitrary injective maps $f_{S,\t{S}} : (S \cup C_s) \times \t{S} \hookrightarrow S \times (\t{S} \cup C_s)$ and $g_S : S \cup C_s \hookrightarrow S \times C_s$.  Observe that the latter is possible since $\card{S}, \card{C_s} \ge 2$, and so $\card{S} + \card{C_s} \le \card{S} \card{C_s}$.  Moreover, unless $\card{S} = \card{C_s} = 2$, $g_S$ can hit at most five-sixths of $S \times C_s$.  Also, given $F \in \cF$, let $A(F) = \cup_{i : F \in \cI_i } C_i$ denote the colours available to $F$ in $\Phi(\vec{C})$, and let $\t{A}(\t{F})$ be the corresponding set of colours under $\t\Phi(\vec{C})$ (with $\t\cI_s$ replacing $\cI_s)$.

We now build an injection $\Phi(\vec{C}) \hookrightarrow \t\Phi(\vec{C})$.  Given a colouring $\vphi \in \Phi(\vec{C})$, we will find a corresponding colouring $\t\vphi \in \t\Phi(\vec{C})$.  We start with sets $F \in \cA$.  As $A(F) = \t{A}(F)$, we simply take $\t\vphi(F) = \vphi(F)$.

A set $\t{F} \in \t\cB$ is only contained in $\t\cI_s$, and hence $\t{A}(\t{F}) = C_s$.  However, it is paired with a set $F = \tau^{-1}(\t{F}) \in \cB$, which is only in $\cI_s$, so $A(F) = C_s$ as well.  We may therefore set $\t{\vphi}(\t{F}) = \vphi(F)$.

The sets in $\cC \sqcup \t\cC$, common to both $\cF$ and $\t\cF$, naturally come in pairs $\{ F, \t{F} \}$, where $F \in \cC$ and $\t{F} = \tau(F) \in \t\cC$.  Let $I = \{ i \in [s-1] : F \in \cI_i \}$ and $S = \cup_{i \in I} C_i$, and define $\t{I}$ and $\t{S}$ similarly for $\t{F}$.  Property (i) implies $\t{I} \subseteq I$, and hence $\t{S} \subseteq S$.  We have $A(F) = S \cup C_s$ and $A(\t{F}) = \t{S}$, while $\t{A}(F) = S$ and $\t{A}(\t{F}) = \t{S} \cup C_s$.  Hence we assign $(\t\vphi(F), \t\vphi(\t{F})) = f_{S,\t{S}}(\vphi(F), \vphi(\t{F}))$.

This leaves us with the sets in $\cD \sqcup \t\cD$.  These again come in pairs $\{ F, \t{F} \}$, where $F \in \cD \subset \cF \cap \t\cF$, while $\t{F} \in \t\cD$ is only in $\t\cF$.  Let $I = \{ i \in [s-1] : F \in \cI_i \}$, and let $S = \cup_{i \in I} C_i$.  We then have $A(F) = S \cup C_s$, while $\t{A}(F) = S$ and $\t{A}(\t{F}) = C_s$.  We may therefore set $(\t\vphi(F), \t\vphi(\t{F})) = g_S(\vphi(F))$.

This completes the definition of $\t\vphi$.  Since this map is defined injectively on each part, this gives us an injection from $\Phi(\vec{C})$ to $\t\Phi(\vec{C})$, showing $\card{\t\Phi(\vec{C})} \ge \card{\Phi(\vec{C})}$.  Now consider the set $F \in \cD$ given by Property (ii).  For this set $F$, $C_1 \subseteq S$, where $S$ is as defined above.  Hence, unless $\card{C_1} = \card{S} = 2$, our injection uses at most five-sixths of the possible colourings of the pair $\{F, \t{F} \}$ in $\t\Phi(\vec{C})$, which would imply $\card{\t\Phi(\vec{C})} \ge \frac65 \card{\Phi(\vec{C})}$, as claimed.
\end{proof}

Our final lemma lets us compare unions of $\floor{r/3}$ and $\ceil{r/3}$ $t$-stars when $r \equiv 1 \pmod 3$.

\begin{lemma} \label{lem:sets4v22}
Suppose $k > t \ge 1$ and $r \ge 7$ is such that $r \equiv 1 \pmod 3$.  Let $s = \ceil{r/3}$, and let $\cI_1, \cI_2, \hdots, \cI_s$ be $t$-stars with centres $T_1, T_2, \hdots, T_s$ respectively, such that $\card{T_i \cap T_j} \le \max(0, 2t - k - 1)$ for all $1 \le i < j \le s$.  Let $\cF = \cI_1 \cup \hdots \cup \cI_{s-1}$ and $\t\cF = \cI_1 \cup \hdots \cup \cI_{s-1} \cup \cI_s$.  If $\vec{C} = (C_1, \hdots, C_{s-1})$ is an optimal partition of $r$ colours into $s-1$ parts, with $\card{C_1} = 4$, let $\vec{C'} = (C_1', C_2, \hdots, C_{s-1}, C_s')$, where $C_1' \cup C_s'$ is some arbitrary partition of $C_1$ into two sets of size two.

If $\Phi(\vec{C})$ is the set of $(r,t)$-colourings of $\cF$ corresponding to $\vec{C}$, and $\t\Phi(\vec{C'})$ the set of $(r,t)$-colourings of $\t\cF$ corresponding to $\vec{C'}$, then $\card{\t\Phi(\vec{C'})} \ge \card{\Phi(\vec{C})}$.
\end{lemma}

\begin{proof}
Fix some arbitrary bijection $f : C_1 \rightarrow C_1' \times C_s'$ with $|C_1'|=|C_s'|=2$.  For a set $F$, let $A(F)$ and $\t{A}(F)$ denote the set of available colours under $\Phi(\vec{C})$ and $\t\Phi(\vec{C'})$ respectively.  Observe that if $k \le 2t -1$, then $\card{T_i \cap T_j} \le 2t - k - 1$ implies that the $t$-stars are pairwise-disjoint.  On the other hand, if $k \ge 2t$, then the centres of the $t$-stars are pairwise-disjoint.

Write $T_1 = \{ x_1, \hdots, x_t \}$ and $T_s = \{ y_1, \hdots, y_t \}$, and as before define the bijection $\tau : \cI_1 \setminus \cI_s \rightarrow \cI_s \setminus \cI_1$ by setting $\tau(F) = F \cup \{ y_i : y_i \notin F \} \setminus \{ x_i : y_i \notin F \}$.  As either the stars or their centres are pairwise-disjoint, replacing elements in $T_1$ with elements in $T_s$ cannot affect membership in other $t$-stars, and hence:
\begin{itemize}
	\item[(i)] for all $F \in \cI_1 \setminus \cI_s$ and $2 \le i \le s-1$, $\tau(F) \in \cI_i$ if and only if $F \in \cI_i$.
\end{itemize}

We now partition our families.  Let $\cG = \cI_2 \cup \hdots \cup \cI_{s-1}$, and set $\cA = (\cI_1 \cap \cI_s) \cup ( \cG \setminus (\cI_1 \cup \cI_s))$.  Let $\cB = (\cI_1 \setminus \cI_s) \setminus \cG$, and set $\t\cB = \tau(\cB)$.  By Property (i), $\t\cB = (\cI_s \setminus \cI_1) \setminus \cG$.  Finally, let $\cC = (\cI_1 \setminus \cI_s) \cap \cG$, and set $\t\cC = \tau(\cC) = (\cI_s \setminus \cI_1) \cap \cG$.  We then have $\cF = \cA \sqcup \cB \sqcup \cC \sqcup \t\cC$, while $\t\cF = \cA \sqcup \cB \sqcup \t\cB \sqcup \cC \sqcup \t\cC$.

We now exhibit an injection $\Phi(\vec{C}) \hookrightarrow \t\Phi(\vec{C'})$.  Given $\vphi \in \Phi(\vec{C})$, an $(r,t)$-colouring of $\cF$ corresponding to $\vec{C}$, we shall injectively build a corresponding $\t\vphi \in \t\Phi(\vec{C'})$.  First consider a set $F \in \cA$.  Since $A(F) = \t{A}(F)$, we may set $\t\vphi(F) = \vphi(F)$.

Next consider the sets in $\cB \sqcup \t\cB$, which naturally come in pairs $\{ F, \t{F} \}$, where $F \in \cI_1$ and $\t{F} = \tau(F) \in \cI_s$.  Thus $A(F) = C_1$, while $\t{A}(F) = C_1'$ and $\t{A}(\t{F}) = C_s'$.  We may therefore set $(\t\vphi(F), \t\vphi(\t{F})) = f(\vphi(F))$.

Finally, we have the sets in $\cC \sqcup \t\cC$, which again can be paired up as $\{F, \t{F} \}$, where $F \in \cC$ and $\t{F} = \tau(F) \in \t\cC$.  Let $I = \{ 2 \le i \le s-1 : F \in \cI_i \}$ and $S = \cup_{i \in I} C_i$.  Using Property (i), we then have $A(F) = S \cup C_1$ and $A(\t{F}) = S$, while $\t{A}(F) = S \cup C_1'$ and $\t{A}(\t{F}) = S \cup C_s'$.  If $\vphi(F) \in C_s' \subset C_1$, set $(\t\vphi(F), \t\vphi(\t{F})) = (\vphi(\t{F}), \vphi(F))$.  Otherwise $\vphi(F) \in \t{A}(F)$, and hence we may set $(\t\vphi(F), \t\vphi(\t{F})) = (\vphi(F), \vphi(\t{F}))$.

As $\t\vphi$ is defined injectively on each part of $\t\cF$, it follows that we have the desired injection from $\Phi(\vec{C})$ into $\t\Phi(\vec{C'})$, thus proving the lemma.
\end{proof}

With these lemmas in place, Proposition~\ref{prop:rcolsets} follows easily.

\begin{proof}[Proof of Proposition~\ref{prop:rcolsets}]
By Lemma~\ref{lem:setsdelta}, $\Delta$ can be made arbitrarily large by taking $n \ge C r^2 k (k-t)$ for some suitably large $C$.  By Theorem~\ref{thm:multicol}, this implies that any optimal family must be the union of $s$ $t$-stars, where $s = \ceil{r/3}$, or perhaps $\floor{r/3}$ if $r \equiv 1 \pmod 3$.  By Corollary~\ref{cor:asymcount}, for any such family $\cF$, if $c(\cF)$ is the number of $(r,t)$-colourings of $\cF$, $c(\cF) = (1 + O(2^{-\Delta})) \sum_{\vec{C} \in \fC} \card{\Phi(\vec{C})}$.  We choose $C$ large enough such that
\[ \frac{19}{20} \sum_{\vec{C} \in \fC} \card{\Phi(\vec{C})} \le c(\cF) \le \frac{21}{20} \sum_{\vec{C} \in \fC} \card{\Phi(\vec{C})}. \]

Now consider an optimal family $\cF = \cI_1 \cup \hdots \cup \cI_s$, where each $\cI_i$ is a $t$-star with centre $T_i$.  Suppose for some $i < j$ we have $\card{T_i \cap T_j} \ge \max(1,2t - k)$, and let $\t\cF$ be the family obtained by replacing $\cI_j$ with a $t$-star $\t\cI_j$ whose centre is pairwise-disjoint from all the others.  Let $\fC' \subseteq \fC$ be the set of partitions of colours where $\card{C_i} \ge 3$ or $\card{C_j} \ge 3$.  Since there can be at most two parts of only two colours, in which case there must be at least three parts in total, we have $\card{\fC'} \ge \frac23 \card{\fC}$.

By Lemma~\ref{lem:setsstarswap}, for every $\vec{C} \in \fC$, we have $\card{\t\Phi(\vec{C})} \ge \card{\Phi(\vec{C})}$, and $\card{\t\Phi(\vec{C})} \ge \frac65 \card{\Phi(\vec{C})}$ for all $\vec{C} \in \fC'$.  Let $\phi_0 = \max_{\vec{C} \in \fC} \card{\Phi(\vec{C})}$ be the largest number of $(r,t)$-colourings of $\cF$ corresponding to a single partition of colours, and let $\phi_1 = \frac{1}{\card{\fC'}} \sum_{\vec{C} \in \fC'} \card{\Phi(\vec{C})}$ be the average number of $(r,t)$-colourings of $\cF$ coming from partitions in $\fC'$.

We consider two cases.  First suppose $\phi_1 \ge \frac56 \phi_0$.  In this case, the total number of $(r,t)$-colourings of $\cF$ can be bounded by $c(\cF) \le \frac{21}{20} \sum_{\vec{C} \in \fC} \card{\Phi(\vec{C})} \le \frac{21}{20} \card{\fC} \phi_0$.  On the other hand, we have
\begin{align*}
c(\t\cF) &\ge \frac{19}{20} \sum_{\vec{C} \in \fC} \card{\t\Phi(\vec{C})} \ge \frac{19}{20} \left( \sum_{\vec{C} \in \fC} \card{\Phi(\vec{C})} + \frac15 \sum_{\vec{C} \in \fC'} \card{\Phi(\vec{C})} \right) \\
	&\ge \frac{19}{20} \left( \frac{20}{21} c(\cF) + \frac15 \card{\fC'} \phi_1 \right) \ge \frac{19}{21} c(\cF) + \frac{19}{100} \left( \frac23 \card{\fC} \right) \left( \frac56 \phi_0 \right) \\
	&\ge \frac{19}{21} c(\cF) + \frac{19}{180} \left( \frac{20}{21} c(\cF) \right) = \frac{190}{189} c(\cF) > c(\cF),
\end{align*}
which contradicts the optimality of $\cF$.  

Hence we must have $\phi_1 < \frac56 \phi_0$.  In this case, we have the bound
\begin{align*}
	c(\cF) &\le \frac{21}{20} \sum_{\vec{C} \in \fC} \card{\Phi(\vec{C})} = \frac{21}{20} \left( \sum_{\vec{C} \notin \fC'} \card{\Phi(\vec{C})} + \sum_{\vec{C} \in \fC'} \card{\Phi(\vec{C})} \right) \\ 	&\le \frac{21}{20} \left( \left( \card{\fC} - \card{\fC'} \right) \phi_0 + \card{\fC'} \phi_1 \right) = \frac{21}{20} \left( \card{\fC} \phi_0 - \card{\fC'} (\phi_0 - \phi_1) \right) \\
	&\le \frac{21}{20} \left( \card{\fC} \phi_0 - \left( \frac23 \card{\fC} \right) \left( \frac16 \phi_0 \right) \right) = \frac{14}{15} \card{\fC} \phi_0.
\end{align*}

Now let $\hat{\cF}$ be the family obtained by successively replacing the $t$-stars $\cI_i$ until their centres obey $\card{T_i \cap T_j} \le \max(0, 2t - k - 1)$ for all $1 \le i < j \le s$.  Since the $t$-stars of $\hat{\cF}$ are pairwise-disjoint when $k \le 2t-2$ and have pairwise-disjoint centres when $k \ge 2t -1$, the symmetry of $\hat{\cF}$ implies that $\hat{\Phi}(\vec{C})$ is the same for every $\vec{C} \in \fC$ (here $\hat{\Phi}(\vec{C})$ denotes the number of $(r,t)$-colourings of $\hat{\cF}$ corresponding to the partition $\vec{C}$).  Moreover, by Lemma~\ref{lem:setsstarswap}, we know $\card{\hat{\Phi}(\vec{C})} \ge \card{\Phi(\vec{C})}$ for every $\vec{C} \in \fC$, and thus $\card{\hat{\Phi}(\vec{C})} \ge \phi_0$ for every $\vec{C} \in \fC$.  Thus 
\[ c(\hat{\cF}) \ge \frac{19}{20} \sum_{\vec{C} \in \fC} \card{\hat{\Phi}(\vec{C})} \ge \frac{19}{20} \card{\fC} \phi_0 > c(\cF), \]
again contradicting the optimality of $\cF$.

Thus if $\cF$ is optimal, its $s$ $t$-stars should be pairwise-disjoint when $k \le 2t- 2$, and should have pairwise-disjoint centres when $k \ge 2t -1$.  When $r \equiv 1 \pmod 3$, though, it remains to determine what $s$ should be.  Let $\cF$ be a family with $\floor{r/3}$ $t$-stars, and let $\t\cF$ be a family with $\ceil{r/3}$ $t$-stars.  By the symmetries of $\cF$ and $\t\cF$, every good partition of the colours gives rise to the same number of $(r,t)$-colourings for each family, say $\phi$ and $\t\phi$ respectively.  By Lemma~\ref{lem:sets4v22}, we know $\t\phi \ge \phi$.  Hence, if $\fC$ is the set of partitions of $r$ colours into $\floor{r/3}$ parts, and $\t\fC$ the set of partitions into $\ceil{r/3}$ parts, we have $c(\cF) \le \frac{21}{20} \card{\fC} \phi$, while $c(\t\cF) \ge \frac{19}{20} \card{\t\fC} \t\phi \ge \frac{19}{20} \card{\t\fC} \phi$.  Finally, setting $s = \floor{r/3}$, observe that $\card{\fC} = \binom{s}{1} \frac{r!}{(3!)^{s-1} 4!}$, while $\card{\t\fC} = \binom{s+1}{2} \frac{r!}{(3!)^{s-1} (2!)^2}$, and so $\card{\t\fC} = 3(s+1) \card{\fC}$.  Thus $c(\t\cF) > c(\cF)$, showing that a family of $\floor{r/3}$ $t$-stars cannot be optimal.

In summary, we find that an optimal family must consist of $\ceil{r/3}$ $t$-stars, with the centres of the stars having pairwise-intersections of size at most $\max(0,2t - k - 1)$.  If $k \ge 2t - 1$, this implies the centres are pairwise-disjoint, which determines the family uniquely up to isomorphism, giving the characterisation of (i) in the statement of Proposition~\ref{prop:rcolsets}.  If $k \le 2t - 2$, this restriction forces the $t$-stars themselves to be pairwise disjoint, as claimed in part (ii) of the proposition.
\end{proof}

When $k \le 2t - 2$, the condition that the $t$-stars be pairwise disjoint does not determine the families uniquely, as there is still some choice in how to distribute the centres.  It does, however, give the asymptotic number of $(r,t)$-colourings, since every good partition will give rise to exactly $\OPT(r)^{\binom{n-t}{k-t}}$ typical colourings, and thus by Corollary~\ref{cor:asymcount}, the total number of $(r,t)$-colourings is $(1 + o(1)) \card{\fC} \OPT(r)^{\binom{n-t}{k-t}}$ as $\Delta \rightarrow \infty$.  The choice of centres will affect the number of atypical colourings, and thus determine the families that are exactly optimal.  Hoppen, Kohayakawa and Lefmann~\cite{HKL12} conjectured that the optimal families should have centres whose pairwise intersections are all of size $2t - k - 1$.

\subsection{Vector spaces} \label{subsec:multicolvectors}

We now return to the vector space setting with the aim of determining which families of vector spaces maximise the number of $(r,t)$-colourings when there are more than three colours available.  Hoppen, Lefmann and Odermann~\cite{HLO16} resolved the problem for $r = 4$, showing that the optimal families are the unions of two $t$-stars whose centres intersect in a $(t-1)$-dimensional subspace.

Drawing parallels to the set family results of Hoppen, Kohayakawa and Lefmann~\cite{HKL12}, they made the following conjecture for $r \ge 5$, which we have rephrased to match our notation.

\begin{conj}[Hoppen--Lefmann--Odermann~\cite{HLO16}] \label{conj:HLOvectors} 
Let $1 \le t < k$ and $r \ge 5$ be fixed integers, $q$ a fixed prime power, and $n$ sufficiently large, and consider $(r,t)$-colourings of $k$-dimensional subspaces of $\Fqn$.
\begin{itemize}
	\item[(i)] If $k \le 2t - 2$, a union of $\ceil{r / 3}$ pairwise-disjoint $t$-stars is asymptotically optimal.
	\item[(ii)] If $k \ge 2t-1$, the number of $(r,t)$-colourings is maximised by a union of $\ceil{r/3}$ $t$-stars whose centres pairwise intersect only in $\vec{0}$.
\end{itemize}
\end{conj}

However, the situation is rather more delicate than for set families, as when $k \ge 2t$, having stars with trivially-intersecting centres does not determine the family up to isomorphism.  In fact, it does not even determine the number of $(r,t)$-colourings asymptotically, and so one must further specify which $t$-dimensional subspaces to use as the centres of the stars.  One might expect the vector space analogue of disjoint sets to be linearly independent subspaces, but this is surprisingly not the case.

In what follows we will apply the results of Section~\ref{subsec:multicolgen} to determine which families of vector spaces asymptotically maximise the number of $(r,t)$-colourings.  In light of the difficulties mentioned above, we shall simplify matters by only considering the case when $r$ is divisible by three, so as to have a cleaner solution to $\OPT(r)$.  Recall also that a key ingredient of our method is the bound on the size of non-extremal maximal families.  For $t = 1$, the result of Blokhuis et al.~\cite{BBCFMPS10} gives the optimal dependence of $n$ on the other parameters, but for $t \ge 2$ we only have the result of Ellis~\cite{E16}, which requires $n$ to be large.  For the sake of brevity, we shall assume $n$ is large enough in both cases, but with careful computation one could obtain an effective bound on $n$ when $t = 1$.  We recall our main result below.

\multicolvs*

Before we proceed with the proof of this proposition, we shall describe the constructions that give us lower bounds on the maximum number of $(r,t)$-colourings, thus providing us with the necessary conditions in the cases above.  Each of the constructions will be a union of $s$ $t$-stars with centres $T_i$, $1 \le i \le s$, and we lower bound the number of $(r,t)$-colourings by fixing some optimal partition of the colours $\vec{C} \in \fC$ assigning three colours to each star.

For the first construction $\cV_1$, take the centres $T_i$ to be linearly independent.  When $k \le 2t - 1$, note that the corresponding stars are pairwise-disjoint, since the span of any two of the centres $T_i$ and $T_j$ is $(2t)$-dimensional, but our subspaces are only $k$-dimensional.  Hence each subspace will have exactly three colours available for it, which shows that the number of $(r,t)$-colourings is at least
\begin{equation} \label{ineq:vslowbd1}
	c(\cV_1) \ge \card{\Phi(\vec{C})} = 3^{\card{\cV_1}} = 3^{s \qbinom{n-t}{k-t}}.
\end{equation}
If $k \ge 2t$, then these stars are no longer disjoint.  However, since any subspace in the intersection of two stars must contain the $(2t)$-dimensional subspace their centres span, we have at most $\binom{s}{2} \qbinom{n-2t}{k-2t}$ subspaces that are in multiple stars.  The remaining subspaces have exactly three colours available, giving
\begin{equation} \label{ineq:vslowbd2}
	c(\cV_1) \ge \card{\Phi(\vec{C})} \ge 3^{s \qbinom{n-t}{k-t} - \binom{s}{2} \qbinom{n-2t}{k-2t}}. 
\end{equation}

For our second construction $\cV_2$, we shall fix some $(2t)$-dimensional subspace $W$.  We now require $T_i \le W$ for each $i$, as well as $T_i \cap T_j = \{ \vec{0} \}$ for all $i \neq j$.  When $q \ge s - 1$, such a collection of subspaces can be built greedily, as we now describe.  Suppose we have built subspaces $T_j$, $1 \le j < i$, and a partial subspace $U_i$ of dimension at most $t-1$.  Extend $U_i$ by a new vector in $W$ that does not lie in $\cup_{j < i} \left( T_j + U_i \right)$.  Each of these spans has dimension at most $2t-1$, with $\vec{0}$ in all of them, so there are fewer than $(s-1)q^{2t-1}$ forbidden vectors, but a total of $q^{2t}$ vectors in $W$.  Hence, provided $q \ge s - 1$, this greedy process can run through to completion.\footnote{For an explicit construction of these centres, consider $W$ as the orthogonal sum of $t$ $2$-dimensional spaces $W_{\ell}$, $1 \le \ell \le t$.  For each $\ell$, let $\{ L_i^{(\ell)} : 1 \le i \le s \}$ be a collection of distinct lines in $W_{\ell}$.  We can then take $T_i$ to be the space spanned by $\{ L_i^{(\ell)} : 1 \le \ell \le t \}$.  }  To lower bound the number of colourings of $\cV_2$, note that the subspaces containing the $(2t)$-dimensional space $W$ have all $r$ colours available to them, while the remaining subspaces each have three colours.  Hence
\begin{equation} \label{ineq:vslowbd3}
	c(\cV_2) \ge \card{\Phi(\vec{C})} = 3^{s \left( \qbinom{n-t}{k-t} - \qbinom{n-2t}{k-2t} \right) } r^{ \qbinom{n-2t}{k-2t} } = 3^{s \qbinom{n-t}{k-t} } \left( s 3^{1-s} \right)^{\qbinom{n-2t}{k-2t}}.
\end{equation}

With these lower bounds in place, we shall now show that any family with at least this many colourings must be as described in Proposition~\ref{prop:multicolvs}.  We begin by showing that the quantity $\Delta$ is large, which will allow us to use Theorem~\ref{thm:multicol}.

\begin{lemma} \label{lem:deltavs}
For fixed $k, r, t$ and $q$, the quantity $\Delta \rightarrow \infty$ as $n \rightarrow \infty$, where $\Delta$ is as defined in~\eqref{defn:multicol}.
\end{lemma}

\begin{proof}
The extremal results of Hsieh~\cite{H75} when $t = 1$ and Frankl and Wilson~\cite{FW86} for $t \ge 2$ show that the largest $t$-intersecting vector spaces are the $t$-stars, hence we may take $\maxfam = \qbinom{n-t}{k-t}$.  Since the intersection of two such stars consists of all vector spaces containing some fixed subspace of dimension at least $t+1$, we have $\twomax = \qbinom{n-t-1}{k-t-1}$.  Regarding the size of the largest non-extremal maximal $t$-intersecting family, the work of Blokhuis et al.~\cite{BBCFMPS10} gives an exact bound when $t = 1$.  However, in our asymptotic setting it will suffice to use the result of Ellis~\cite{E16} that holds for all $t$ provided $n$ is large enough, giving $\hmsize \le \left(1 + O(q^{-n}) \right) \qbinom{k+1}{1} \qbinom{n-t-1}{k-t-1}$.  Finally, as discussed in Section~\ref{subsec:3colvectors}, the result of Balogh et al.~\cite{BDDLS15} gives $\nummax \le \qbinom{n}{k}^{\binom{2(k-t)+1}{k-t}}$.

Substituting these parameters into~\eqref{defn:multicol}, and using our bounds from~\eqref{ineq:qbounds}, we find
\begin{align*}
	\Delta &= \left( \lg 9 - 3 \right) \maxfam - r \max \left( \hmsize, \twomax \right) - \frac{r \twomax}{3} \lg \OPT(r) - r \lg \nummax \\
	&\ge \left( \lg 9 - 3 \right) \qbinom{n-t}{k-t} - \left( \left(1 + O(q^{-n}) \right) r \qbinom{k+1}{1} + s^2 \lg 3 \right) \qbinom{n-t-1}{k-t-1} \\
	&\qquad - r \binom{2(k-t)+1}{k-t} \lg \qbinom{n}{k} \\
	&\ge (\lg 9 - 3) q^{(k-t)(n-k)} - 5 \left( r \qbinom{k+1}{1} + s^2 \lg 3 \right) q^{(k-t-1)(n-k)} - 2^{2(k-t)+1} r k (n-k) \lg q \\
	&= \Omega\left(q^{(k-t)(n-k)} \right) \rightarrow \infty. \qedhere
\end{align*}
\end{proof}

We are now ready to prove our main result of this section.

\begin{proof}[Proof of Proposition~\ref{prop:multicolvs}]
By Lemma~\ref{lem:deltavs}, we have $\Delta \rightarrow \infty$.  In particular, Theorem~\ref{thm:multicol} implies that any asymptotically optimal family $\cV$ must be the union of $s$ $t$-stars with centres $T_i$, $1 \le i \le s$, say.  By Corollary~\ref{cor:asymcount} we also know that the number of $(r,t$)-colourings of any such family is $(1 + o(1)) \sum_{\vec{C} \in \fC} \card{\Phi(\vec{C})}$.  Since $r$ is divisible by three, every partition in $\fC$ assigns exactly three colours to each of the $t$-stars, and hence $\card{\Phi(\vec{C})}$ is independent of $\vec{C}$.  Thus we deduce that the number of $(r,t)$-colourings is $(1 + o(1)) \card{\fC} \card{\Phi(\vec{C})}$, where $\vec{C}$ is any equipartition of the colours over the $t$-stars.

For each $V \in \cV$, let $m(V)$ represent the number of the $t$-stars that contain $V$ or, equivalently, the number of centres $T_i$ contained in $V$.  If $a(V)$ is the number of colours available for the subspace $V$, we then have $a(V) = 3 m(V) \le 3^{m(V)}$, with equality if and only if $m(V) = 1$. Thus 
\begin{equation} \label{eqn:vsnumcolourings} \card{\Phi(\vec{C})} = \prod_{V \in \cV} a(V) = 3^{\sum_{V \in \cV} m(V)} \prod_{V \in \cV} \frac{3m(V)}{3^{m(V)}} = 3^{s \qbinom{n-t}{k-t}} \prod_{V \in \cV} m(V) 3^{1-m(V)},
\end{equation}
where each factor in the final product is at most $1$.

If $k \le 2t-1$, and the $t$-stars are not disjoint, then there is some $V \in \cV$ for which $m(V) \ge 2$, and hence $m(V) 3^{1 - m(V)} \le 2/3$.  This implies that $\card{\Phi(\vec{C})} \le 2 \cdot 3^{s \qbinom{n-t}{k-t} - 1}$.  On the other hand, as shown in~\eqref{ineq:vslowbd1}, if the $t$-stars are disjoint, then we always have $\card{\Phi(\vec{C})} = 3^{ s \qbinom{n-t}{k-t}}$.  Hence it follows that any asymptotically optimal familly must have pairwise-disjoint $t$-stars, and any such family has an asymptotically equal number of $(r,t)$-colourings, completing the proof of (i).

If $k \ge 2t$, suppose we have $\dim(T_1 \cap T_2) \ge 1$.  This implies that the span of $T_1 \cup T_2$ is at most $(2t-1)$-dimensional, and hence there are at least $\qbinom{n-2t+1}{k-2t+1}$ subspaces for which $m(V) \ge 2$.  By~\eqref{eqn:vsnumcolourings} and~\eqref{ineq:qbounds}, we then have 
\[ \card{\Phi(\vec{C})} 3^{-s \qbinom{n-t}{k-t}} \le \left(\frac23 \right)^{\qbinom{n-2t+1}{k-2t+1}} \le \left( \frac23 \right)^{q^{(k-2t+1)(n-k)}}, \]
while~\eqref{ineq:vslowbd2} gives the significantly larger lower bound
\[ \card{\Phi(\vec{C})} 3^{-s \qbinom{n-t}{k-t}} \ge 3^{- \binom{s}{2} \qbinom{n-2t}{k-2t}} \ge 3^{- 2 s^2 q^{(k-2t)(n-k)}}.\]
Hence for every asymptotically optimal family we must have $T_i \cap T_j = \{ \vec{0} \}$ for all $i \neq j$, proving (ii).

Finally, for more precise results, we consider the number of subspaces for which $m(V) \ge 2$.  In order for $V$ to be contained in multiple $t$-stars, it must contain one of the $(2t)$-dimensional spaces spanned by a pair of the centres $T_i$.  Let $\{ U_1, U_2, \hdots, U_{\ell} \}$ be the set of distinct $(2t)$-dimensional subspaces spanned by pairs of centres, where $1 \le \ell \le \binom{s}{2}$.

Note that the span of any pair of spaces $U_j$ is at least $(2t+1)$-dimensional, and hence at most $\binom{\ell}{2} \qbinom{n-2t-1}{k-2t-1}$ subspaces $V \in \cV$ will contain two or more of these spaces.  Hence for each $j$, there are at least $\qbinom{n-2t}{k-2t} - \binom{\ell}{2} \qbinom{n-2t-1}{k-2t-1} = (1 - o(1)) \qbinom{n-2t}{k-2t}$ subspaces $V \in \cV$ that contain $U_j$ but do not contain $U_{j'}$ for any $j' \neq j$.  If $m_j$ is the number of centres $T_i$ contained in $U_j$, these subspaces containing only $U_j$ will have $m(V) = m_j$.

Using~\eqref{eqn:vsnumcolourings} gives
\begin{equation} \label{ineq:vsprecise}
	\card{\Phi(\vec{C})} 3^{-s \qbinom{n-t}{k-t}} \le \prod_{j=1}^{\ell} \left( m_j 3^{1- m_j } \right)^{(1 - o(1))\qbinom{n-2t}{k-2t}}.
\end{equation}

Note that every pair of centres spans exactly one of the subspaces $U_j$, so we must have $\sum_{j=1}^{\ell} \binom{m_j}{2} = \binom{s}{2}$.  The following optimisation result, to be proven later, bounds how large the right-hand side can be in the above inequality.

\begin{claim} \label{clm:vsopt}
Given integers $s \ge 2$, $1 \le \ell \le \binom{s}{2}$, and $2 \le m_j \le s$, $1 \le j \le \ell$, satisfying the constraint $\sum_{j = 1}^{\ell} \binom{m_j}{2} = \binom{s}{2}$, we have $\sum_{j=1}^{\ell} \left( m_j - 1 - \log_3 m_j  \right) \ge s - 1 - \log_3 s$, with equality if and only if $\ell = 1$ and $m_1 = s$.
\end{claim}

Now suppose we have at least two distinct $(2t)$-dimensional subspaces $U_j$.  By~\eqref{ineq:vsprecise}, we must have $\log_3 \left( \card{ \Phi(\vec{C})} 3^{ - s \qbinom{n-t}{k-t}} \right) \le (1 - o(1)) \qbinom{n-2t}{k-2t} \sum_{j=1}^{\ell} (\log_3 m_j + 1 - m_j )$.  Since $\ell \ge 2$, Claim~\ref{clm:vsopt} implies this is strictly smaller than $\qbinom{n-2t}{k-2t} \left( \log_3 s - 1 - s \right)$.  Comparing this to the lower bound from~\eqref{ineq:vslowbd3}, valid whenever $q \ge s-1$, we find that $\cV$ cannot be asymptotically optimal.

Hence any asymptotically optimal family $\cV$ must be the union of $s$ $t$-stars with centres $T_i$, where $T_i \cap T_j = \{ \vec{0} \}$ for all $i \neq j$, and all the centres lie in a single $(2t)$-dimensional space.  In this case the calculation in~\eqref{ineq:vslowbd3} holds, which shows that all such constructions have asymptotically the same number of $(r,t)$-colourings, and hence are asymptotically optimal, completing the proof.
\end{proof}

We complete this section by proving Claim~\ref{clm:vsopt}.

\begin{proof}[Proof of Claim~\ref{clm:vsopt}]
Consider the function
\[ g(x) = \frac{x - 1 - \log_3 x}{\binom{x}{2}} = \frac{2}{x} - \frac{2 \log_3 x}{x (x-1)}, \]
for which we have 
\[ g'(x) = \frac{2(2x-1) \log_3 x}{x^2 (x-1)^2} - \frac{2}{x^2} - \frac{2}{x^2 (x-1) \ln 3}. \]

It can be checked that $g'(x) < 0$ for all $x \ge 2$, and hence for $2 \le x \le s$,
\[ g(x) \ge  g(s) = \frac{s - 1 - \log_3 s}{\binom{s}{2}}, \]
with equality if and only if $x = s$.

Hence we have
\[ \sum_{j=1}^{\ell} \left( m_j - 1 - \log_3 m_j \right) = \sum_{j=1}^{\ell} \binom{m_j}{2} g(m_j) \ge g(s) \sum_{j=1}^{\ell} \binom{m_j}{2} = g(s) \binom{s}{2} = s - 1 - \log_3 s,  \]
with equality if and only if $m_j = s$ for all $j$.  This in turn implies $\ell = 1$, as required.
\end{proof}

%%%%%%%%%%%%%%%%%%%%%%%%%%%%%%%%%%%%%%%%
%%
%%  SECTION: CONCLUSION
%%
%%%%%%%%%%%%%%%%%%%%%%%%%%%%%%%%%%%%%%%%

\section{Concluding remarks} \label{sec:conc}

In this paper we have studied the Erd\H{o}s--Rothschild problem for intersecting families, combining extremal and stability results with strong bounds on the number of maximal intersecting families.  This approach allowed us to unify and improve previous results in this direction, often obtaining sharp dependencies on the size of the ground set with respect to the other parameters.

Theorem~\ref{thm:3col} provides a simple sufficient condition that ensures the extremal intersecting families are precisely those with the maximum number of $(3,t)$-colourings.  We then showed that this condition can be verified in the context of permutations, vector spaces and set families, leading to three-coloured Erd\H{o}s--Rothschild results for these problems.  The same method can give results in other settings as well; for instance, one can bound the number of three-colourings of set families without large monochromatic matchings, but we omitted the similar derivation for the sake of (relative) brevity.

The same approach can be used for larger numbers of colours, resulting in Theorem~\ref{thm:multicol}, a stability result stating that the optimal families are unions of extremal families.  For set families, this shows that the results of Hoppen, Kohayakawa and Lefmann~\cite{HKL12} continue to hold when the size of the ground set is quadratic in the uniformity of the family.  For families of vector spaces, we were able to characterise those families with an asymptotically optimal number of $(r,t)$-colourings when $r$ is divisible by three, thus partially resolving a conjecture of Hoppen, Lefmann and Odermann~\cite{HLO16}.

We now highlight some possible avenues for further investigation.

\paragraph{Sharper bounds for three colours}

One of the advantages of our method is the sharp bounds on the relative sizes of the parameters.  For instance, with $t$-intersecting $k$-uniform set families over $[n]$, where $3 \le t \le k-3$, we could show that the $t$-stars are the only optimal families as soon as we exceed the Erd\H{o}s--Ko--Rado threshold; that is, once $n \ge (t+1)(k-t+1) + 1$.

However, the most natural case is undoubtedly when $t = 1$, for which we require $n \ge 3k + 10 \ln k$.  The same result, though, could be true for $n$ as small as $2k + 1$.  To obtain results when $n \approx 2k$ may require more careful study of large maximal intersecting families, or the use of other combinatorial results.

Another interesting problem would be to determine the optimal families \emph{below} the Erd\H{o}s--Ko--Rado threshold.  When $n < (t+1)(k-t+1)$, is the number of $(3,t)$-colourings still maximised by the largest $t$-intersecting families, or do we observe qualitatively different behaviour in this range?

\paragraph{Multicoloured set families of larger uniformity}

Our results are somewhat weaker when the number of colours increases.  Within the context of ($1$-)intersecting set families, we required $n \ge C r^2 k (k-1)$ for some absolute constant $C$.  This quadratic threshold commonly arises when studying intersecting families, as when $k$ is larger, there are non-trivial intersecting families comparable in size to the stars.  However, there is no reason to believe that the result fails to hold: are the stars still optimal when $k$ is, say, linear in $n$?

One bottleneck in our argument was ruling out uniquely-appearing maximal families in the proof of Claim~\ref{clm:multiplicities}, where we used a rather expensive union bound.  Some more careful estimates here could go some way towards obtaining a condition in Theorem~\ref{thm:multicol} that would still hold for set families of larger uniformities.

Moreover, referring to Proposition~\ref{prop:rcolsets}, let us emphasise again that for $k\le 2t-2$, it is still a challenging open problem to give a full characterisation of all optimal families, as has already been asked for by Hoppen, Kohayakawa and Lefmann~\cite{HKL12}.   

\paragraph{Exact results for multicoloured vector spaces}

In the setting of vector spaces, we were able to show that it is asymptotically optimal to take a union of $t$-stars, where the centres of these stars are pairwise trivially intersecting, yet are contained in a fixed $2t$-dimensional subspace.  When $r = 6$, this simply means that we have two linearly independent $t$-spaces, and all such choices are isomorphic.  However, for larger $r$, this in general does not specify the families up to isomorphism, and so it remains to determine which families exactly maximise the number of $(r,t)$-colourings.

When $t = 1$ and $r = 9$, the centres of the stars are three lines in a two-dimensional space.  Since the group of all automorphisms of $\Fq^2$ acts $3$-transitively on the set of lines, and preserves the dimensions of the intersections of subspaces, it follows that any such family has the same number of $(9,1)$-colourings, and we can again describe the optimal families exactly.

However, for larger $t$, or larger $r$, we do not know if families satisfying the condition can in fact have different numbers of $(r,t)$-colourings, nor do we have any conjecture as to which families should be optimal.

\paragraph{A general Erd\H{o}s--Rothschild result}

Ultimately, one would like to have a general solution to the Erd\H{o}s--Rothschild problem: for which extremal problems is it true that the extremal structures maximise the number of two- or three-colourings without monochromatic copies of the forbidden substructure?  The proofs of our general results do not explicitly require that the extremal problem in question concerns intersecting families, but we know no other problem that enjoys both such strong stability results and so few maximal constructions.  For instance, the inequality~\eqref{ineq:3col} would not be satisfied for Mantel's Theorem, even though the Erd\H{o}s--Rothschild extension is known to hold.  Can one develop a stronger sufficient condition for the Erd\H{o}s--Rothschild problem?  On the other hand, it is known that for some extremal problems, the trivial lower bound for the Erd\H{o}s--Rothschild extension is not correct even for two or three colours (see, for instance,~\cite{HKL14}).  It would be of great interest to determine which features of an extremal problem imply that the extremal constructions are also optimal for the Erd\H{o}s--Rothschild problem.

\paragraph{Acknowledgements}  The authors would like to thank D\'aniel Kor\'andi for helpful conversations, of which one led to the formulation of the optimisation problem solved in Claim~\ref{clm:vsopt}.  We are also grateful to the anonymous referees for their suggestions for improving the presentation of our paper.

\end{document}